\documentclass[11pt,a4paper]{article}

\usepackage{amsmath}
\usepackage{amsfonts}
\usepackage{amsthm}
\usepackage[pdftex]{color,graphicx}
\usepackage{graphicx}
\usepackage{color}
\usepackage[colorlinks,linkcolor=blue,citecolor=red,urlcolor=black]{hyperref}
\usepackage{esvect}

\newtheorem{Theorem}{Theorem}

\newtheorem{Proposition}{Proposition}
\newtheorem{Lemma}{Lemma}
\newtheorem{Remark}{Remark}

\newcommand{\dt}{{\mathrm{d}}t}
\newcommand{\ds}{{\mathrm{d}}s}
\newcommand{\dx}{{\mathrm{d}}x}

\newcommand{\Id}{{\mathbf{1}}}
\newcommand{\dom}{{\mathrm{dom}~}}

\begin{document}

\title{Spectral analysis on ruled surfaces with combined Dirichlet and Neumann boundary conditions}
\author{Rafael T. Amorim and Alessandra A. Verri}

\date{\today}

\maketitle 

\begin{abstract}
Let $\Omega$ be an unbounded two dimensional strip on a ruled surface in $\mathbb{R}^d$, $d\geq2$.
Consider the Laplacian operator in $\Omega$ with Dirichlet and Neumann boundary conditions on opposite sides of $\Omega$.
We prove some results on the existence
and absence of the discrete spectrum of the operator; which are influenced by the twisted and bent effects 
of $\Omega$. 
Provided that $\Omega$ is thin enough, we show an asymptotic behavior of the eigenvalues. 
The interest in those considerations lies on the difference from the purely Dirichlet case. Finally,
we perform an appropriate dilatation in $\Omega$ and we compare the results.
\end{abstract}

\

\noindent {\bf Mathematics Subject Classification (2020).}  Primary: 81Q10; Secondary: 35P15, 47F05, 58J50.

\

\noindent    {\bf Keywords:} 
Laplacian operator; Ruled surfaces; Dirichlet and Neumann boundary conditions.
\



\section{Introduction}

Let $\Omega$ be an unbounded two dimensional strip on a ruled surface in $\mathbb{R}^d$, $d\geq 2$, and consider the Laplacian operator restricted to $\Omega$.
On the boundary $\partial\Omega$, suppose the Dirichlet or Neumman conditions, or a combination of these ones. The spectrum of this class of operators
has been
extensively studied in the last years 
\cite{bori1,bori2,bulla,duclosfk,dittrichakriz1,dittrichakriz2,duclos,pavelduclos,exnerseba,exnertater,solomyak,friedsol,gold,davidsurf,daviddn,
	davidruled,davidkriz,davidlu,davidmainpaper,renger}. In particular,  
the existence of discrete eigenvalues depends on 
the geometry of $\Omega$ and the boundary conditions on $\partial \Omega$. 
For the model studied in this paper, we consider a
combination of the Dirichlet and  Neumann  conditions on $\partial \Omega$.

At first, suppose that $\Omega$ is an unbounded straight strip in $\mathbb{R}^2$.
In this case, it is known that the spectrum of the Dirichlet (resp. Neumann) Laplacian operator in $\Omega$
is purely essential. However, the spectral analysis of the Laplacian operator 
is a non trivial problem 
if we consider combined Dirichlet and Neumann  conditions on different parts of 
the boundary $\partial \Omega$ \cite{bulla,dittrichakriz1, kovarik}.
For example, the authors of \cite{bulla} proved the existence of discrete spectrum for the Laplacian operator in $\Omega = \mathbb R \times (0,1)$ with 
Neumann conditions on the segment 
$(a, b) \times \{1\}$, $-\infty < a < b < \infty$, and Dirichlet conditions on 
$\partial \Omega \backslash \{(a, b) \times \{1\}\}$. 
By taking $\Omega = \mathbb R \times (0,d)$, $0 < d < \infty$, one of the results of \cite{dittrichakriz1}
ensures the existence of a discrete eigenvalue for the Laplacian in $\Omega$ with  
Dirichlet conditions on $\{(x,d); |x|>a\}$, for any $a>0$,
and  Neumann  conditions on $\partial\Omega\backslash\{(x,d); |x|\geq a\}$.
On the other hand, in \cite{kovarik}, the main result is a Hardy-type inequality for the
Laplacian in $\mathbb{R}\times(-d,d)$, subject to
Dirichlet  conditions on the segments
$(-\infty, 0) \times \{-d\}$ and $(0, \infty) \times \{d\}$, 
and  Neumann  conditions on the segments
$(-\infty, 0) \times \{d\}$ and $(0, \infty) \times \{-d\}$.

Now, let $T$ be a  lower bounded self-adjoint operator.  Denote by  $\{\lambda_j(T)\}_{j \in \mathbb N}$ the non-decreasing sequence of numbers corresponding to the
spectrum of $T$ according to the minimax principle;
each $\lambda_j(T)$ represents a discrete eigenvalue that is less or equal to the threshold of the essential spectrum of $T$.
Throughout the text, 
$-\Delta_\mathbb R$ denotes the one-dimensional Laplacian operator on $\mathbb R$;
$\Id$ denotes the identity operator.

Let $\Gamma : \mathbb R \to \mathbb R^2$ 
be a plane curve parameterized by its arc-length $s$ and denote by $\gamma(s)$ its signed curvature at $s$; $\Gamma$ is supposed to satisfy some regularity conditions which will not be detailed in this text.  
Consider the case where $\Omega$ is an unbounded planar strip obtained by moving the segment $(0,\varepsilon)$, $\varepsilon>0$, along $\Gamma$ with respect to its normal vector field.
At first, let $-\Delta_\Omega^D$ be the Dirichlet Laplacian in $\Omega$. 
If $\Omega$ is asymptotically straight (i.e., $\gamma(s) \to 0$, as $|s| \to \infty$), 
then $(\pi/ \varepsilon)^2$ is the threshold of the essential spectrum of 
$-\Delta_\Omega^D$.
In addition,  if  $\gamma(s)\neq 0$, then the discrete spectrum 
is nonempty. If $\Omega$ is thin enough, i.e., if $\varepsilon$ is small enough, one has
\begin{equation}\label{diintbac}
	\lambda_j(-\Delta_\Omega^D) = \left(\frac{\pi}{\varepsilon}\right)^2 + \lambda_j\left(-\Delta_\mathbb R - \frac{\gamma^2}{4} \Id\right) + O(\varepsilon).
\end{equation}
However, the Neumann Laplacian, denoted by $-\Delta_\Omega^N$, has a purely essential spectrum, given by $[0, \infty)$, even in the case that $\gamma \neq 0$.
If $\varepsilon$ is small enough, then
\begin{equation}
	\lambda_j(-\Delta_\Omega^N) = \lambda_j\left(-\Delta_\mathbb R\right) + O(\varepsilon).
\end{equation}
See, for example, \cite{pavelduclos,exnerseba,gold,davidkriz} for more on these results.
Now, let $-\Delta_\Omega^{DN}$ be the Laplacian operator with 
Dirichlet  conditions in $\Gamma(\mathbb R)$ and Neumann  conditions on the opposite boundary.
This situation was studied in \cite{dittrichakriz2,daviddn,davidkriz}.
If $\Omega$ is asymptotically straight, $(\pi/2 \varepsilon)^2$ is the threshold of the essential spectrum.
Under the condition that $\gamma$ has a compact support, the results of \cite{dittrichakriz2} ensure that if
$\int_\mathbb R \gamma(s) \ds< 0$, then the discrete spectrum of the operator is nonempty. On the other hand, if
$\gamma(s) \geq 0$, for all $s \in \mathbb R$, then there are no  discrete eigenvalues.
In \cite{davidkriz}, the authors obtained similar results with less restrictions on  $\gamma$.
In particular, they  proved that if $\Omega$ is asymptotically straight, $\gamma\neq0$, 
$\gamma \in L^1(\mathbb R)$, and $\int_\mathbb R \gamma(s) \ds\leq 0$, then
the discrete spectrum of the operator is nonempty.
In addition, they showed that the number of eigenvalues can be arbitrarily large if
$\Omega$ is thin enough.
In \cite{daviddn}, the author found that, for $\varepsilon$ small enough,
\begin{equation}\label{asyintdnd}	
\lambda_j(-\Delta_\Omega^{DN}) =  \left(\frac{\pi}{2 \varepsilon}\right)^2 + 
\lambda_j\left(-\Delta_\mathbb R + \frac{\gamma}{\varepsilon} \Id \right) + O(1)
 = \left(\frac{\pi}{2 \varepsilon}\right)^2 + \frac{\inf \gamma}{\varepsilon} + o(\varepsilon^{-1}).
\end{equation}
As discussed in that work, this expansion ensures the existence of discrete eigenvalues if $\Omega$ is asymptotically straight
and $\gamma(s)$ assumes a negative value; the number of eigenvalues increases to infinity as $\varepsilon$ approaches zero.
Note that, locally, the condition
$\gamma < 0$ means that the ratio of the curvature of the Neumann boundary to the Dirichlet one is the biggest.

Now, let $\Omega$ be a two dimensional strip (surface) in $\mathbb R^d$, $d \geq 2$, obtained by translating a segment along a curve with respect to 
an appropriate frame. 
A model of this type was previously introduced in
\cite{davidmainpaper} where the authors performed a detailed spectral study of the Dirichlet Laplacian in $\Omega$;
as we will discuss later.
Inspired by \cite{dittrichakriz2, daviddn, davidkriz, davidmainpaper}, our purpose is to study the spectral problem of the Laplacian operator with
Dirichlet and Neumann conditions on opposite sides of a two dimensional strip as the one defined in \cite{davidmainpaper}.
In the next paragraphs, we present more details of the problem.

Let $\Gamma: \mathbb R \to \mathbb R^{n+1}$, $n \geq 1$, be a curve of class $C^2$
parameterized by its arc-length $s$, i.e., $|\Gamma'(s)|=1$, for all $s \in \mathbb R$.
The vector $T(s):=\Gamma'(s)$ denotes its unitary tangent vector at the point  $\Gamma(s)$ and the number $k(s):=|\Gamma''(s)|$, $s \in \mathbb R$,
is called the curvature of $\Gamma$ at the position $\Gamma (s)$. 
It is known  that  a Frenet frame for the curve $\Gamma$ does not need to
exist. However, 
in Appendix A of \cite{davidmainpaper}, the authors proved the existence of a {\it relatively parallel adapted frame for 
	$\Gamma$}, which consists of $n$ normal vector fields $N_1, \cdots, N_n$ of class $C^1$ that satisfy
\begin{equation}\label{frameint}
	\left(
	\begin{array}{c}
		T \\
		N_1 \\
		\vdots \\
		N_n
	\end{array}\right)' 
	=
	\left(
	\begin{array}{cccc}
		0      & k_1    & \cdots & k_n \\
		-k_1   & 0      & \cdots & 0   \\
		\vdots & \vdots &        & \vdots \\
		-k_n   & 0      & \cdots &  0
	\end{array}\right)
	\left(
	\begin{array}{c}
		T \\
		N_1 \\
		\vdots \\
		N_n
	\end{array}\right),
\end{equation}
where $k_j:\mathbb R \to \mathbb R$, $j \in \{1, \cdots, n\}$, are continuous  functions, and
$k_1^2+ \cdots + k_n^2=k^2$. 
Throughout this paper, we are going to adopt the adapted frame given by 
(\ref{frameint}).

Let $\Theta_j: \mathbb R \to \mathbb R$, $j \in \{1, \cdots, n\}$, be functions of class $C^1$ that satisfy
\begin{equation}\label{condthetainnorm}
	\Theta_1^2 + \cdots + \Theta_n^2 = 1.
\end{equation}
Define the funtion $\Theta: \mathbb R \to \mathbb R^{n}$ by $\Theta:= (\Theta_1, \cdots, \Theta_n)$, and  the normal field
\begin{equation}\label{normalfield}
	N_\Theta: = \Theta_1 N_1 + \cdots + \Theta_n N_n.
\end{equation}
Let $\varepsilon>0$ be a real number. Consider the map
\begin{equation}\label{ldefinition}
	\begin{array}{rcll}
		{\cal L}_\varepsilon: & \mathbb R^2 & \longrightarrow        &  \mathbb R^{n+1}, \\
		& (s,t)       & \longmapsto    & \Gamma(s) + \varepsilon t N_\Theta(s).
	\end{array}
\end{equation}
Finally, define the strip
\begin{equation}\label{maindomain}
	\Omega_\varepsilon = \mathcal{L}_\varepsilon\left(\mathbb R \times (0,1)\right).
\end{equation}
Roughly speaking, $\Omega_\varepsilon$ is obtained by translating the segment $(0,\varepsilon)$ along $\Gamma$ with respect to a normal field
(\ref{normalfield}).
As in \cite{davidmainpaper}, 
$\Theta$ is called {\it twisting vector};
if $\Theta'=0$, then $\Omega_\varepsilon$ is called 
{\it untwisted} or {\it purely bent} strip; if $k \cdot \Theta := k_1 \Theta_1 + \cdots + k_n \Theta_n = 0$,  then
$\Omega_\varepsilon$ is called 
{\it unbent} or {\it purely twisted} strip.
Geometrically, interpreting $\Gamma$ as a curve in $\Omega_\varepsilon$, 
$k \cdot \Theta$ is the {\it geodesic curvature} of $\Gamma$, and
$-|\Theta'|^2/f_\varepsilon$ is the {\it Gauss curvature} of $\Omega_\varepsilon$; $f_\varepsilon$ is given by 
(\ref{geodgaussdef}) in Section \ref{strip}.
A detailed geometric description of $\Omega_\varepsilon$ (including figures of several possible cases) can be found in  \cite{davidmainpaper}.

Now, let $-\Delta_{\Omega_\varepsilon}^{DN}$ be the Laplacian operator in $\Omega_\varepsilon$ with Dirichlet and Neumann conditions on $\Gamma(\mathbb{R}) = \mathcal{L}_\varepsilon(\mathbb{R}\times\{0\})$ and $\mathcal{L}_\varepsilon(\mathbb{R}\times\{1\})$, respectively.
It is defined as the self-adjoint operator associated with the quadratic form
\begin{equation}\label{qfiqldef}
	a_\varepsilon(\varphi) = \int_{\Omega_\varepsilon} |\nabla \varphi|^2 \dx, 
	\quad
	\dom a_\varepsilon := \left\{\varphi\in H^1(\Omega_\varepsilon); \varphi = 0 \textrm{ on } \Gamma(\mathbb{R}) \right\};
\end{equation}
$\nabla \varphi$ denotes the gradient of $\varphi$ corresponding to the metric induced by the immersion ${\cal L}_\varepsilon|_{\mathbb R \times (0,1)}$.
For simplicity, we denote $-\Delta_\varepsilon^{DN} := -\Delta_{\Omega_\varepsilon}^{DN}$.

The goal of the first part of this paper is to study the spectral problem of
$-\Delta_\varepsilon^{DN}$. 
For technical reasons, we assume that
\begin{equation}\label{mainassuption}
	k \cdot \Theta 
	\in L^\infty(\mathbb{R}) ,
	\quad \hbox{and} \quad \varepsilon \|k \cdot \Theta\|_{L^\infty(\mathbb{R})} < 1.
\end{equation}
In this situation, we find the essential spectrum and we discuss conditions that ensure the existence or the absence of the discrete 
spectrum for $-\Delta_\varepsilon^{DN}$. As in \cite{davidmainpaper}, 
we analyze how the geometry of $\Omega_\varepsilon$ affects the results.

Denote by $-\Delta_{(0,\varepsilon)}^{DN}$ the Laplacian operator acting in $L^2(0,\varepsilon)$ with Dirichlet and Neumann conditions on $\{0\}$ and $\{\varepsilon\}$, respectively; its first eigenvalue is $(\pi/2\varepsilon)^2$.
In the particular case where 
$k\cdot\Theta = 0$ and $\Theta' = 0$, the region $\Omega_\varepsilon$ is a flat strip in the sense that the metric induced by ${\cal L}_\varepsilon$ is Euclidean; see Section \ref{strip}.
Thus, $-\Delta_\varepsilon^{DN}$ coincides with the Laplacian operator in $\mathbb{R}\times(0,\varepsilon)$ with Dirichlet and Neumann conditions on $\mathbb{R}\times\{0\}$ and $\mathbb{R}\times\{\varepsilon\}$, respectively.
In this situation, it is known that the spectrum of $-\Delta_\varepsilon^{DN}$ is purely essential and
coincides with the interval $[(\pi/2\varepsilon)^2,\infty)$.
In this work, we assume that 
\begin{equation}\label{asymptoticallyflat}
	\lim_{|s|\to \infty} (k\cdot\Theta)(s) = 0 \quad \textrm{ and } \quad \lim_{|s|\to \infty} |\Theta'(s)| = 0,
\end{equation}
where $|\Theta'| := (\Theta_1^{'2} + \cdots + \Theta_n^{'2})^{1/2}$.
Since the conditions in (\ref{asymptoticallyflat}) imply that
the bent and twisted effects locally slow
down at the infinity, it is expected the 
stability of the essential spectrum in that situation. This is confirmed by the 
following result.

\begin{Theorem}\label{sigmaess}
	Assume the conditions $(\ref{condthetainnorm})$, $(\ref{mainassuption})$, and $(\ref{asymptoticallyflat})$. Then,
	\[
	\sigma_{ess} (-\Delta_\varepsilon^{DN}) = [(\pi/2 \varepsilon)^2,\infty).
	\]
\end{Theorem}

The proof of this theorem follows the arguments of the proof of 
Theorem 3.1 of \cite{davidmainpaper};  it will be omitted in this text.
Our next result provides sufficient conditions to the existence of discrete eigenvalues for 
$-\Delta_\varepsilon^{DN}$ below $(\pi/2 \varepsilon)^2$.

\begin{Theorem}\label{sigmadis}
	Assume the condition {\rm (\ref{condthetainnorm})}. If $k\cdot\Theta = 0$ and $\Theta'\neq 0$, then
	\[
	\inf \sigma(-\Delta_\varepsilon^{DN}) < \left(\frac{\pi}{2\varepsilon}\right)^2.
	\]
\end{Theorem}

The proof of this result is presented in Section \ref{twistedstrip}.

Under the assumptions of Theorems \ref{sigmaess} and \ref{sigmadis} one has 
that the discrete spectrum of $-\Delta_\varepsilon^{DN}$ is nonempty. Then, we can say that
the twisted effect can create discrete eigenvalues for the operator.
At this point, there is a natural question: 
can the bending effect also create discrete eigenvalues? 
To answer this question, at first,
define the function $r:[0,4/5)\longrightarrow\mathbb{R}$,
\begin{equation}\label{gamma}
	r(x):=
	\frac{x^2(2-x)}{4(1-x)^2(4-5x)}.
\end{equation}
Let $x_0\in (0,4/5)$ be so that $r(x_0) = (\pi/2)^2$.
Then, we present the following
result.

\begin{Theorem}\label{hardy1}
	Assume the conditions {\rm (\ref{condthetainnorm})} and  {\rm (\ref{mainassuption})}. 
	If $\Theta'= 0$ and $k\cdot\Theta \neq 0$ satisfies $k\cdot\Theta \geq 0$,
	$ \varepsilon \|k\cdot\Theta\|_{L^\infty(\mathbb{R})}\leq x_0$, 
	then there exists a positive constant $c$ such that
	\begin{equation}\label{hardyequation}
		-\Delta_\varepsilon^{DN} -\left(\frac{\pi}{2\varepsilon}\right)^2 \geq c \rho,
	\end{equation}
	where $\rho(s):= 1/(1+s^2)$.
\end{Theorem}

Theorem \ref{hardy1} deals with the case where $\Omega_\varepsilon$ is a purely bent strip. 
If an appropriate twisted effect is added, we obtain the following version.

\begin{Theorem}\label{hardy2}
	Assume the conditions {\rm (\ref{condthetainnorm})} and  {\rm (\ref{mainassuption})}. Let $k$ and $\Theta$ be so that
	%
	\[k\cdot\Theta\neq 0, \quad k\cdot\Theta \geq 0, \quad 
	\varepsilon \|k\cdot\Theta\|_{L^\infty(\mathbb{R})}\leq x_0, \quad |\Theta'(s)|\leq \frac{\delta}{1+s^2},\]
	for some $\delta\geq 0$. 
	Then,
	\[-\Delta_\varepsilon^{DN} > \left(\frac{\pi}{2\varepsilon}\right)^2,\]
	for all $\delta$ enough small.
\end{Theorem}

Under the assumptions of  Theorems \ref{sigmaess} and \ref{hardy1} (resp. \ref{hardy2}) one has 
the absence of the discrete spectrum for $-\Delta_\varepsilon^{DN}$.
Theorem \ref{hardy2} shows that
the twisted effect is not a sufficient condition to create discrete eigenvalues for the operator.
The proofs of Theorems \ref{hardy1} and \ref{hardy2} are presented in Section \ref{bentstrip}.

In Section 4 of \cite{daviddn},  possible extensions for the results
presented there was discussed. In particular, the case of 
strips embedded in an abstract two-dimensional Riemannian manifold was mentioned and the
expected result in that situation was announced. In this text,
we will provide details of the in case of ruled strips.

To study the model where $\Omega_\varepsilon$ is thin enough, i.e., $\varepsilon > 0$ is small enough, we add the assumptions
\begin{equation}\label{Thetalimited}
	(k\cdot\Theta)', |\Theta'|, |\Theta''|\in L^\infty(\mathbb{R}).
\end{equation}
Based on arguments of \cite{daviddn}, in  Section \ref{thinstrip}, we 
prove the following result.

\begin{Theorem}\label{autovalores}
	Assume the conditions {\rm (\ref{condthetainnorm})}, {\rm (\ref{mainassuption})}, and {\rm (\ref{Thetalimited})}. 
	If $k\cdot\Theta\neq 0$, then, for each $j \in \mathbb N$, 
	\begin{equation}\label{asympbetweff}
		\lambda_j(-\Delta_{\varepsilon}^{DN}) 
		=
		\left(\frac{\pi}{2\varepsilon}\right)^2 + \lambda_j \left(-\Delta_\mathbb R  + \frac{k\cdot\Theta}{\varepsilon} \Id \right) + O(1) 
		=
		\left(\frac{\pi}{2\varepsilon}\right)^2 
		+ \frac{\inf (k\cdot\Theta)}{\varepsilon} 
		+ o(\varepsilon^{-1}),
	\end{equation}
for all $\varepsilon>0$ small enough.
\end{Theorem}

Note that (the twisted effect)
$\Theta'$ does not influence the asymptotic behavior given by (\ref{asympbetweff}).
As in \cite{daviddn}, it
can be used to ensure
the existence of discrete spectrum for $-\Delta_\varepsilon^{DN}$. In fact,
in addition to the assumptions of  Theorem \ref{autovalores}, assumes the conditions in (\ref{asymptoticallyflat}).
In this case, one has $\inf \sigma_{ess} (-\Delta_\varepsilon^{DN}) = (\pi/2 \varepsilon)^2$.
If $k \cdot \Theta$ assumes a negative value, then, for each $j \in \mathbb N$, the equality (\ref{asympbetweff}) 
implies the existence of discrete eigenvalues for $-\Delta_\varepsilon^{DN}$. Furthermore, the number of these eigenvalues
can be arbitrarily large since $\Omega_\varepsilon$ is thin enough.
As also noted in \cite{daviddn, davidbohemica}, 
the existence of discrete spectrum is ruled by a local phenomenon, i.e., by extreme points where the curve ${\cal L}_\varepsilon(\mathbb R \times \{1\})$
is locally bigger than ${\cal L}_\varepsilon(\mathbb R \times \{0\})$.

\begin{Remark}\label{remintoy}{\rm
		Let $\Sigma$ be a connected orientable $C^2$ hypersurface in $\mathbb R^d$, $d \geq 2$, whose orientation is 
		ruled by a globally defined unit normal vector field $n: \Sigma \to {\mathbb S}^{d-1}$.
		For each $\varepsilon > 0$ small enough, define the tubular neighborhood
		${\cal T}:=\{x+\varepsilon t n(x) \in \mathbb R^d; (x,t) \in \Sigma \times (0,1)\}$.
		In \cite{davidbohemica}, was considered the Laplacian operator subject to Dirichlet and Neumann
		conditions on $\Sigma$ and $\Sigma_\varepsilon := \Sigma + \varepsilon n (\Sigma)$, respectively. Using the notation $-\Delta_{\cal T}^{DN}$ for this operator,
		in that work, it was found that
		\begin{equation}
			\lambda_j(-\Delta_{\cal T}^{DN}) = \left(\frac{\pi}{2 \varepsilon}\right)^2 + \lambda_j \left( -\Delta_\Sigma + \frac{k_M}{\varepsilon} \Id\right) + O(1),
		\end{equation}
		where $-\Delta_\Sigma$ denotes the Laplace-Beltrami operator on $\Sigma$, subject to Dirichlet 
		conditions if $\partial\Sigma$ is not empty, and $k_M$ is a $d-1$ multiple of  the mean curvature of $\Sigma$; the asymptotic expansion is local and depends on  
		the extreme points where the ratio of the area of the Neumann boundary
		to the Dirichlet one is locally the biggest.
	}
\end{Remark}

When $\Omega_\varepsilon$ is a purely twisted strip, 
we get the following result.

\begin{Theorem}\label{asymtwisted}
Assume the conditions {\rm (\ref{condthetainnorm})} and {\rm (\ref{Thetalimited})}. 
If $k\cdot\Theta= 0$, then, for each $j \in \mathbb N$, 
\begin{equation*}
	\lambda_j(-\Delta_{\varepsilon}^{DN}) 
	=
	\left(\frac{\pi}{2\varepsilon}\right)^2 
	+ \lambda_j \left(-\Delta_\mathbb R  - \frac{|\Theta'|^2}{2} \Id \right) 
	+ O(\varepsilon),
\end{equation*}
for all $\varepsilon>0$ small enough.
\end{Theorem}

The proof of Theorem \ref{asymtwisted} is presented in Section \ref{thinstrip}.


\begin{Remark}{\rm
In \cite{davidmainpaper}, the authors  studied the Dirichlet Laplacian restricted to the symmetric strip 
${\cal S}:= {\cal L}_\varepsilon(\mathbb{R}\times(-1,1))$, where ${\cal L}_\varepsilon$ is given by (\ref{ldefinition}).
In particular, under the conditions in (\ref{asymptoticallyflat}), they proved that the assumptions 
$k\cdot\Theta \neq 0$ and $\Theta'= 0$
ensure the existence of discrete spectrum for the operator. 
However, when $k \cdot \Theta =0$ (resp. $|(k \cdot \Theta)(s)| \leq \delta/(1+ s^2)$, for some $\delta > 0$ small enough), 
and  $\Theta' \neq 0$ satisfies $\varepsilon \sup|\Theta'| \leq \sqrt{2}$, 
they obtained 
Hardy-type inequalities for the operator. In that work, 
the geodesic curvature acts as an attractive interaction spectrum in the sense that it diminishes the 
spectrum and creates discrete eigenvalues. The Gauss curvature acts as
a repulsive interaction in the sense that it induces Hardy-type inequalities. Following the same terminologies of \cite{davidmainpaper},
for the model studied in this work,
as a consequence of Theorems \ref{sigmaess}, \ref{sigmadis}, \ref{hardy1}, and \ref{hardy2},
the  Gauss curvature acts as an attractive interaction;
if the geodesic curvature is positive, then it acts as a repulsive interaction; if the geodesic curvature is negative, then 
it acts as a attractive interaction.}
\end{Remark}

Theorem \ref{autovalores} is obtained as a consequence of the ``convergence'' in the norm resolvent sense of the family of operators $\{-\Delta_\varepsilon^{DN}\}_{\varepsilon>0}$.
Due to the $\varepsilon$ dependency of the operator $-\Delta_\mathbb R  + (k\cdot\Theta/\varepsilon) \Id$, 
we do not have an effective operator in the usual sense, as $\varepsilon$ approaches to zero.
This phenomenon was also noted in \cite{daviddn, davidbohemica}, and can be seen as
a characteristic due to the boundary conditions of the model.
In the next paragraphs, we present an adaptation of the problem which
eliminates the $\varepsilon$ dependency of the 
effective operator.

For each $\varepsilon > 0$ small enough, let $\Theta_\varepsilon:\mathbb{R}\longrightarrow \mathbb{R}^n$ be a function of class $C^1$ defined by
$$\Theta_\varepsilon(s) := \Theta\left(\frac{s}{\sqrt{\varepsilon}}\right), \quad s \in\mathbb{R};$$
by using the notation $\Theta_\varepsilon(s) = (\Theta_1^\varepsilon, \cdots, \Theta_n^\varepsilon)$, due to (\ref{condthetainnorm}), one has
$(\Theta_1^\varepsilon)^2 + \cdots + (\Theta_n^\varepsilon)^2 = 1$.
With respect to the reference curve $\Gamma$, denote by 
$\Gamma_\varepsilon$ the curve of class $C^2$ whose curvature $k_\varepsilon$ is given by
$$k_\varepsilon(s) := k\left(\frac{s}{\sqrt{\varepsilon}}\right), \quad s\in\mathbb{R}.$$
The normal vector fields of $\Gamma_\varepsilon$ are denoted by $N_1^\varepsilon, \cdots, N_n^\varepsilon$,
and $N_{\Theta_\varepsilon}^\varepsilon:= \Theta_1^\varepsilon N_1^\varepsilon + \cdots + \Theta_n^\varepsilon N_n^\varepsilon$. 

Consider the mapping
\begin{equation*}
	\begin{array}{rcll}
		\widetilde{\cal L}_\varepsilon: & \mathbb R^2 & \longrightarrow        &  \mathbb R^{n+1}, \\
		& (s,t)       & \longmapsto    & \Gamma_\varepsilon(s) +   \varepsilon t N_{\Theta_\varepsilon}^\varepsilon(s),
	\end{array}
\end{equation*}
and the thin strip
\begin{equation*}
	\widetilde{\Omega}_\varepsilon := \widetilde{\cal L}_\varepsilon (\mathbb{R}\times(0,1)).
\end{equation*}

Let $-\Delta_{\widetilde{\Omega}_\varepsilon}^{DN}$ be the Laplacian operator in $\widetilde{\Omega}_\varepsilon$ with Dirichlet and Neumann conditions on $\Gamma_\varepsilon(\mathbb{R})$ and $\widetilde{\mathcal{L}}_\varepsilon(\mathbb{R}\times\{1\})$, respectively, i.e., it 
is the self-adjoint operator associated with the quadratic form
\begin{equation}\label{qtilde}
	\tilde{a}_\varepsilon(\varphi) = \int_{\widetilde{\Omega}_\varepsilon} |\nabla \varphi|^2 \dx, 
	\quad 
	\dom \tilde{a}_\varepsilon = \{\varphi\in H^1(\widetilde{\Omega}_\varepsilon); \varphi = 0 \textrm{ on } \Gamma_\varepsilon(\mathbb{R}) \}.
\end{equation}
For simplicity, write $-\tilde{\Delta}_\varepsilon^{DN} : = -\Delta_{\widetilde{\Omega}_\varepsilon}^{DN}$.
In Section \ref{twistedeffect}, we prove the following result.

\begin{Theorem}\label{lambdaassintotico}
Assume the conditions {\rm (\ref{condthetainnorm})}, {\rm (\ref{mainassuption})}, and {\rm (\ref{Thetalimited})}.  Then, 
for each $j\in\mathbb{N}$,
\begin{equation}\label{vardepasympbe}
	\lambda_j(-\tilde{\Delta}_\varepsilon^{DN}) =
	\left(\frac{\pi}{2\varepsilon}\right)^2
	+ \frac{1}{\varepsilon}\lambda_j\left(-\Delta_\mathbb R  + \left(k\cdot\Theta - \frac{|\Theta'|^2}{2}
	\right) \Id\right)
	+O(1),
\end{equation}
for $\varepsilon>0$ small enough.
\end{Theorem}

As in Theorem \ref{autovalores},  Theorem \ref{lambdaassintotico} is obtained from the convergence in the norm 
resolvent sense of the 
operators $\{-\tilde{\Delta}_\varepsilon^{DN}\}_{\varepsilon > 0}$.
We emphasize that the effective operator in (\ref{vardepasympbe}) does not depend on
the parameter $\varepsilon$. 
Furthermore, 
define $V(s):= (k\cdot\Theta)(s) - |\Theta'(s)|^2/2$ and assume the conditions in (\ref{asymptoticallyflat}).
As a consequence of the asymptotic behavior given by (\ref{vardepasympbe}),
if there 
exists a  real number $s_0 > 0$ so that $\int_{-s_0}^{s_0} V(s) \ds < 0$,
then $V(s)$ acts as a repulsive contribution as $\varepsilon \to 0$.

This paper is organized as follows. In Section \ref{strip} we present some details of the construction of the region in 
(\ref{maindomain}) and we make usual change of coordinates 
in the quadratic form in (\ref{qfiqldef}).
Section \ref{twistedstrip} is dedicated to study the discrete spectrum of $-\Delta_\varepsilon^{DN}$ when $\Omega$ is a purely twisted strip. 
In Section \ref{bentstrip} we prove some results on the absence of discrete spectrum of the operator as well as the proofs of Theorems  \ref{hardy1} and \ref{hardy2}.
In Sections \ref{thinstrip} and \ref{twistedeffect} we find an asymptotic behavior for the eigenvalues of $-\Delta_\varepsilon^{DN}$ and  $-\tilde{\Delta}_\varepsilon^{DN}$, respectively.
In Appendix \ref{appendix001} results that are useful in this text are presented.

\section{Geometry of the region and change of coordinates}\label{strip}

Recall the mapping ${\cal L}_\varepsilon$ and the quadratic form $a_\varepsilon(\varphi)$ are given by (\ref{ldefinition}) and (\ref{qfiqldef}), respectively, in the Introduction. 
In this section, we identify  $\Omega_\varepsilon = {\cal L}_\varepsilon(\Lambda)$ with the 
Riemannian manifold $(\Lambda, {\cal G}_\varepsilon)$, where  $\Lambda := \mathbb R \times (0,1)$ and ${\cal G}_\varepsilon$ is given by (\ref{geodgaussdef}), below.
After that, we perform a usual change of coordinates in the quadratic form $a_\varepsilon(\varphi)$.

Define the metric ${\cal G}_\varepsilon:= \nabla {\cal L}_\varepsilon \cdot (\nabla {\cal L}_\varepsilon)^\perp$.
Some calculations show that 
\begin{equation}\label{geodgaussdef}
	{\cal G}_\varepsilon=\left(\begin{array}{cc}
		f_\varepsilon^2  & 0 \\
		0  & \varepsilon^2
	\end{array}\right), \quad
	f_\varepsilon(s,t):= \sqrt{(1 - \varepsilon t (k\cdot\Theta)(s))^2 + \varepsilon^2 t^2 |\Theta'(s)|^2 }.
\end{equation}
Let ${\cal J}_\varepsilon$ be the Jacobian matrix of ${\cal L}_\varepsilon$. One has
$\det {\cal J}_\varepsilon = | \det {\cal G}_\varepsilon|^{1/2} =\varepsilon f_\varepsilon > 0$, for all $(s, t) \in \Lambda$.
Since $\Gamma$ and $\Theta$ are smooth functions, 
the inverse function theorem implies that
the map ${\cal L}_\varepsilon: \Lambda \to \Omega_\varepsilon$ is a local smooth diffeomorphism. 
In addition, assume that ${\cal L}_\varepsilon$ is injective. Thus, 
the strip $\Omega_\varepsilon$  does not self-intersect and it is interpreted as an immersed submanifold
in $\mathbb R^{n+1}$.
As a consequence, the
map ${\cal L}_\varepsilon: \Lambda \to \Omega_\varepsilon$ is a global smooth diffeomorphism.
Therefore, $(\Lambda, {\cal G}_\varepsilon)$ is an abstract Riemannian manifold.

Now, we perform 
a change of coordinates 
so that the quadratic form $a_\varepsilon(\varphi)$ starts to act in the 
Hilbert space ${\cal H}_\varepsilon := L^2(\Lambda,f_\varepsilon\ds\dt)$ instead of $L^2(\Omega_\varepsilon)$. Consider the unitary operator
\[\begin{array}{rccc}
	{\cal U}_\varepsilon: & L^2(\Omega_\varepsilon) & \longrightarrow        &  L^2(\Lambda,f_\varepsilon\ds\dt), \\
	& \psi    & \longmapsto &  \psi \circ {\cal L}_\varepsilon,
\end{array}\]
and define the quadratic form
\begin{align}
	b_\varepsilon (\psi)  
	:= & \;
	a_\varepsilon \left({\cal U}_\varepsilon^{-1} \psi \right) 
	= \int_\Lambda \langle \nabla \psi , {\cal G}_\varepsilon^{-1} \nabla \psi \rangle f_\varepsilon \ds \dt \nonumber\\	               
	= & 
	\int_\Lambda \frac{|\partial_s \psi |^2}{f_\varepsilon} \ds \dt 
	+ \frac{1}{\varepsilon^2}\int_\Lambda | \partial_t \psi  |^2 f_\varepsilon \ds\dt, 
	\label{formb}
\end{align}
$\dom b_\varepsilon := {\cal U}_\varepsilon (\dom a_\varepsilon)$;
$\partial_t =\partial/\partial_t$ and $\partial_s =\partial/\partial_s$.
In particular, if $|\Theta'|\in L^\infty(\mathbb{R})$,
one has  
$\dom b_\varepsilon = \{\psi\in H^1 (\Lambda); \psi(s,0)=0\textrm{ a.e. } s\in\mathbb{R}\}.$

\section{Existence of discrete eigenvalues} \label{twistedstrip}
 
This section is dedicated to prove Theorem \ref{sigmadis}. 


 
\begin{proof}[\bf Proof of Theorem \ref{sigmadis}]
Let $\varphi\in C_0^\infty(\mathbb{R})$ be a real-valued function so that
$0 \leq \varphi \leq 1$, $\varphi = 1$ on $[-1,1]$, and $\varphi=0$ on $\mathbb{R}\backslash(-2,2)$. 
For each $n\in\mathbb{N}$, define
\[\psi_n(s,t) := \varphi_n(s)\chi_1(t),\]
where $\varphi_n(s):=\varphi(s/n)$ and 
\begin{equation}\label{chidefinition}
	\chi_1(t) := \sqrt{2}\sin\left(\frac{\pi t}{2} \right).
\end{equation}
In particular, $\chi_1$ is the eigenfunction corresponding to the first eigenvalue of the Laplacian operator in $(0,1)$ with Dirichlet and Neumann boundary conditions on
$\{0\}$ and $\{1\}$, respectively. Note that $(\psi_n)_{n \in \mathbb N} \subset \dom b_\varepsilon$.

Recall the assumptions  $k\cdot\Theta =0$ and $\Theta'\neq0$. 
Some calculations show that
\begin{equation*}
b_\varepsilon (\psi_n) - \left(\frac{\pi}{2\varepsilon}\right)^2 \|\psi_n\|^2_{{\cal H}_\varepsilon} 
=
\int_\Lambda \frac{|\varphi_n'\chi_1|^2}{f_\varepsilon} \ds \dt 
- \int_\Lambda \frac{|\Theta'(s)|^2}{f_\varepsilon} |\varphi_n|^2 t \chi_1 \chi_1'\ds \dt.
\end{equation*}

Since $\varphi_n \to 1$ pointwise, as $n\to \infty$, and 
\[\int_\Lambda \frac{|\varphi_n'\chi_1|^2}{f_\varepsilon} \ds \dt \leq \frac{1}{n}\int_\mathbb{R} |\varphi'|^2 \ds,\]
for all $n\in\mathbb{N}$, one has
\begin{equation}\label{rightside}
\lim_{n\to\infty} \left[b_\varepsilon (\psi_n) - \left(\frac{\pi}{2\varepsilon}\right)^2 \|\psi_n\|^2_{\cal H_\varepsilon}\right] 
=
- \int_\Lambda \frac{|\Theta'(s)|^2}{f_\varepsilon} t \chi_1 \chi_1'\ds \dt.
\end{equation}

The limit of the right side of (\ref{rightside}) is negative. In fact, 
\[\frac{|\Theta'(s)|^2}{f_\varepsilon} t \chi_1 \chi_1' = \frac{|\Theta'(s)|^2}{f_\varepsilon} 
\pi t \sin\left(\frac{\pi t}{2}\right) \cos\left(\frac{\pi t}{2}\right),\]
is a non-trivial, non-negative function.
Then, it is sufficient to take $n_0$ large enough so that
\[b_\varepsilon (\psi_{n_0}) - \left(\frac{\pi}{2\varepsilon}\right)^2 \|\psi_{n_0}\|^2_{\cal H_\varepsilon} < 0,\]
and the result follows by minimax principle.
\end{proof}

\section{Absence of discrete spectrum}\label{bentstrip}

This section is dedicated to prove Theorems \ref{hardy1} and \ref{hardy2}. 
At first, we assume that $\Theta' =0$. Under this condition, the function $f_\varepsilon(s,t)$ defined in (\ref{geodgaussdef}) is reduced to
$
h_\varepsilon(s,t) :=1 - \varepsilon t(k\cdot\Theta)(s)$. Consider the quadratic form
\begin{equation*}
c_\varepsilon (\psi) := \int_\Lambda \frac{|\partial_s \psi |^2}{h_\varepsilon} \ds \dt 
+ \frac{1}{\varepsilon^2}\int_\Lambda  | \partial_t \psi  |^2 h_\varepsilon \ds\dt,
\end{equation*}
$\dom c_\varepsilon:= \{\psi\in H^1 (\Lambda); \psi(s,0)=0\textrm{ a.e. } s\in\mathbb{R}\}$,
acting in the Hilbert space ${\cal H}'_\varepsilon:= L^2(\Lambda, h_\varepsilon\ds\dt)$.

To prove Theorem \ref{hardy1}, for each $s\in\mathbb{R}$, we define the auxiliary one-dimensional self-adjoint operator
\[S_\varepsilon (s) := -\partial_t^2 + \varepsilon \frac{(k\cdot\Theta)(s)}{h_\varepsilon(s,t)} \; \partial_t,\qquad
\dom S_\varepsilon (s) := \{v\in H^2(0,1);v(0)=v'(1)=0\},\]
acting in the Hilbert space $L^2((0,1),h_\varepsilon(s, \cdot)\dt)$;
the case  $\varepsilon = 0$ corresponds to the Laplacian operator $-\Delta_{(0,1)}^{DN}$ in $L^2(0,1)$.
Denote by $\lambda_0^\varepsilon(s)$ the first eigenvalue of $S_\varepsilon(s)$. 

By minimax principle,  we have the estimate
\begin{equation*}\label{clower}
	c_\varepsilon(\psi) -\left(\frac{\pi}{2\varepsilon}\right)^2 \int_\Lambda |\psi|^2 h_\varepsilon \ds \dt
	\geq  
	\int_\Lambda \left(\frac{\lambda_0^\varepsilon(s)}{\varepsilon^2} - \left(\frac{\pi}{2\varepsilon}\right)^2\right) |\psi|^2h_\varepsilon\ds\dt.
\end{equation*}
Then,
\begin{equation}\label{eq17}
-\Delta_\varepsilon^{DN} -\left(\frac{\pi}{2\varepsilon}\right)^2 \geq \frac{\lambda_0^\varepsilon(s)}{\varepsilon^2} - \left(\frac{\pi}{2\varepsilon}\right)^2,
\end{equation}
in the quadratic form sense.

Recall the point $x_0$ defined by $r(x_0)=(\pi/2)^2$, where $r$ is given by (\ref{gamma}) in the Introduction. 
We start with the following lemma.

\begin{Lemma}\label{eigenvaluepositive}
Assume that $k\cdot\Theta\neq 0$ satisfies $k\cdot\Theta\geq 0$ and $\varepsilon\|k\cdot\Theta\|_{L^\infty(\mathbb{R})}\leq x_0$. Then, 
$\lambda_{0}^\varepsilon(s) - (\pi/2)^2$ is a non-trivial, non-negative function.
\end{Lemma}
\begin{proof}
Fix $s\in\mathbb{R}$. 
Consider the unitary operator 
\[\begin{array}{rccc}
	{\cal V}: & L^2(0,1) & \longrightarrow        &  L^2((0,1),  h_\varepsilon(s, \cdot) \dt) \\
	& v    & \longmapsto &  h_\varepsilon^{-1/2} v
\end{array},\]
and define the self-adjoint operator $\widetilde{S}_\varepsilon(s) := {\cal V}^{-1} S_\varepsilon(s) {\cal V}$. 
More precisely,   
\[\widetilde{S}_\varepsilon(s)  = -\partial_{t}^2  - \frac{\varepsilon^2(k\cdot\Theta)(s)^2}{4 h_\varepsilon(s,t)^2} \Id,\]
\[\dom \widetilde{S}_\varepsilon(s) = \left\{v\in H^2(0,1);v(0)=0, v'(1)+\frac{\varepsilon(k\cdot\Theta)(s)}{2h_\varepsilon(s,1)}v(1)=0\right\}.\]
In particular, denote by $\tilde{u}_0^\varepsilon(s, \cdot)$ the normalized eigenfunction of $\widetilde{S}_\varepsilon(s)$ corresponding to the eigenvalue $\lambda_0^\varepsilon(s)$.

Now, define the  operator $Y_\varepsilon(s) := \partial_t^2$, $\dom Y_\varepsilon(s) = \dom \widetilde{S}_\varepsilon(s)$. 
Denote by  $\nu_0^\varepsilon(s)$  its first eigenvalue and by $v_0^\varepsilon(s)$ the corresponding (real) eigenfunction. 

Lemma \ref{lemmakriz} in Appendix \ref{appendix001} and the assumption $k\cdot\Theta \geq  0$ 
ensure that
\begin{align}
\left(\frac{\pi}{2}\right)^2 
& \leq
\nu_0^\varepsilon(s) - \frac{\varepsilon(k\cdot\Theta)(s)}{2h_\varepsilon(s,1)}\, \frac{v_0^\varepsilon(s,1)^2}{\|v_0^\varepsilon(s, \cdot)\|_{L^2(0,1)}^2}\nonumber\\
& \leq 
\lambda_0^\varepsilon(s) - \frac{\varepsilon(k\cdot\Theta)(s)}{2h_\varepsilon(s,1)}\, \frac{v_0^\varepsilon(s,1)^2}{\|v_0^\varepsilon(s, \cdot)\|_{L^2(0,1)}^2} 
+ \int_0^1 \frac{\varepsilon^2(k\cdot\Theta)(s)^2}{4 h_\varepsilon(s,t)^2} |\tilde{u}_0^\varepsilon(s, t)|^2\dt\nonumber\\
& \leq 
\lambda_0^\varepsilon(s) - \frac{\varepsilon(k\cdot\Theta)(s)}{2h_\varepsilon(s,1)}\, \frac{v_0^\varepsilon(s,1)^2}{\|v_0^\varepsilon(s, \cdot)\|_{L^2(0,1)}^2} 
+ \frac{\varepsilon^2(k\cdot\Theta)(s)^2}{4 h_\varepsilon(s,1)^2}\label{inequalitylambda}.
\end{align}

A straightforward calculation shows that
$v_0^\varepsilon(s,t) = \sin(\sqrt{\nu_0^\varepsilon(s)}t)$ and $\nu_0^\varepsilon(s)$ satisfies 
\[\sqrt{\nu_0^\varepsilon(s)}=- \frac{\varepsilon (k\cdot\Theta)(s)}{2h_\varepsilon(s,t)} \tan\left(\sqrt{\nu_0^\varepsilon(s)}\right), \quad
\hbox{and}
\quad \nu_0^\varepsilon(s)\in((\pi/2)^2,\pi^2).\]
Then, 
\[\frac{v_0^\varepsilon(s,1)^2}{\|v_0^\varepsilon(s, \cdot)\|_{L^2(0,1)}^2} 
=\frac{\sin^2\left(\sqrt{\nu_0^\varepsilon(s)}\right)}{\frac{1}{2}-\frac{\sin \left(2\sqrt{\nu_0^\varepsilon(s)}\right)}{4\sqrt{\nu_0^\varepsilon(s)}}}
=\frac{2\nu_0^\varepsilon(s)}{\alpha(s)^2+\alpha(s)+\nu_0^\varepsilon(s)},\]
where $\alpha(s):= \varepsilon (k\cdot\Theta)(s)/2h_\varepsilon(s,1)$.

Now, the strategy is to show that
\begin{equation}\label{strategy01}
\alpha(s)\frac{2\nu_0^\varepsilon(s)}{\alpha(s)^2+\alpha(s)+\nu_0^\varepsilon(s)} -\alpha(s)^2\geq 0.
\end{equation}
Since $r$ is a increasing function, 
\begin{equation}\label{lastestimate}
	\alpha(s)^2\frac{1+\alpha(s)}{2-\alpha(s)} = r(\varepsilon(k\cdot\Theta)(s))\leq r(x_0)= \left(\frac{\pi}{2}\right)^2 \leq \nu_0^\varepsilon(s),
\end{equation}
for all $s\in\mathbb{R}$. 
Then, (\ref{strategy01}) follows from (\ref{lastestimate}).
Finally, the inequalities (\ref{inequalitylambda}) and (\ref{strategy01}) imply that  $\lambda_0^\varepsilon(s)-(\pi/2)^2\geq 0$. 

It remains to prove that $\lambda_0^\varepsilon(s)-(\pi/2)^2$ is a non-trivial function.
Suppose that $\lambda_0^\varepsilon(s)-(\pi/2)^2= 0$. 
Using (\ref{inequalitylambda}) and (\ref{strategy01}) together with the assumption $k\cdot\Theta \geq  0$, one has $\alpha(s)=0$.
Then, $k\cdot\Theta = 0$, which is not allowed.
\end{proof}

Lemma \ref{eigenvaluepositive} shows that 
(\ref{eq17}) is a Hardy-type inequality whenever the assumptions of Theorem \ref{hardy1} hold true. 
However, the function $\lambda_0^\varepsilon(s) -(\pi/2)^2$ might not be positive everywhere in $\Lambda$. 
As in \cite{davidmainpaper}, we call (\ref{eq17}) a {\it local} Hardy inequality.
The next step is to transfer it into the {\it global} Hardy inequality (\ref{hardyequation}) of Theorem \ref{hardy1}.

\begin{proof}[\bf Proof of Theorema \ref{hardy1}]
	
Lemma \ref{eigenvaluepositive} ensures that $\lambda_0^\varepsilon(s)-(\pi/2)^2$ is a
non-negative, non-trivial function. Then, there exists a bounded open interval $I\subset \mathbb{R}$ so that $\lambda_0^\varepsilon(s)-(\pi/2)^2 > 0$,
for all $s \in I$.
Define the operator
\[Z_\varepsilon := -\partial_s^2 + \left(\frac{\lambda_0^\varepsilon(s)}{\varepsilon^2} - \left(\frac{\pi}{2\varepsilon}\right)^2\right) \Id,\]
acting in $L^2(I)$, where the functions in $\dom Z_\varepsilon$ satisfy the Neumann boundary conditions on $\partial I$.
Denote by $\mu_0^\varepsilon > 0$ its first eigenvalue.

Now, define $C_1:= 1-\varepsilon \|k\cdot\Theta\|_{L^\infty(\mathbb{R})}$.
For each $\psi\in \dom c_\varepsilon$, we obtain
\begin{align}\label{interpolation1}
	c_\varepsilon(\psi)-\left(\frac{\pi}{2\varepsilon}\right)^2 \|\psi\|_{{\cal H}'_\varepsilon}^2
	&\geq 
	\int_\Lambda |\partial_s \psi|^2 \ds\dt + 
	\int_\Lambda \left(\frac{\lambda_0^\varepsilon(s)}{\varepsilon^2} - \left(\frac{\pi}{2\varepsilon}\right)^2\right) |\psi|^2 h_\varepsilon\ds\dt\nonumber\\
	&\geq 
	C_1 \int_{I\times (0,1)} \left(|\partial_s \psi|^2 + \left(\frac{\lambda_0^\varepsilon(s)}{\varepsilon^2} - \left(\frac{\pi}{2\varepsilon}\right)^2\right) |\psi|^2\right)\ds\dt\nonumber\\
	& \geq
	C_1 \mu_0^\varepsilon \int_{I\times (0,1)} |\psi|^2 \ds\dt.
\end{align}
Furthermore, since $\lambda_0^\varepsilon(s)-(\pi/2)^2\geq 0$, one has
\begin{equation}\label{lowerlimitation}
	c_\varepsilon(\psi) - \left(\frac{\pi}{2\varepsilon}\right)^2\|\psi\|_{{\cal H}'_\varepsilon}^2
	\geq 
	\int_\Lambda|\partial_s \psi|^2\ds\dt.
\end{equation}
Let $s_0$ be the middle point of $I$. Let $\eta\in C^\infty(\mathbb{R})$ be such that $0\leq \eta\leq 1$, $\eta =0$ in a neighbourhood of $s_0$, and $\eta=1$ in $\mathbb{R}\backslash I$. 
For $\psi\in \dom c_\varepsilon$, write 
$\psi = \eta\psi+(1-\eta)\psi$. Then,
\begin{align}
\int_\Lambda \frac{|\psi|^2}{1+(s-s_0)^2}\ds \dt 
&\leq 
2 \int_\Lambda \frac{|\eta\psi|^2}{(s-s_0)^2}\ds\dt + 2 \int_\Lambda |(1-\eta)\psi|^2\ds\dt \nonumber \\
& \leq 
8 \int_\Lambda|\partial_s(\eta\psi)|^2\ds\dt + 2 \int_{I\times (0,1)} |\psi|^2\ds\dt\nonumber\\
& \leq 
16 \int_\Lambda|\partial_s \psi|^2\ds\dt 
+ 16 C_2 \int_{I\times (0,1)} |\psi|^2\ds\dt, \label{interpolation2}
\end{align}
where $C_2:=\|\eta'\|_{L^\infty(\mathbb{R})}^2 + 1/8$; in the second estimate was used
the classical Hardy inequality 
$4 \int_\mathbb{R}|\varphi'|^2\dx \geq \int_\mathbb{R} (|\varphi|^2/x^2)\dx$,  $\varphi\in H_0^1(\mathbb{R}\backslash\{0\})$.

As a consequence of (\ref{lowerlimitation}) and (\ref{interpolation2}),
\begin{equation} \label{interpolation3}
	c_\varepsilon(\psi) - \left(\frac{\pi}{2\varepsilon}\right)^2\|\psi\|_{{\cal H}'_\varepsilon}^2
	\geq
	\frac{1}{16} \int_\Lambda \frac{|\psi|^2}{1+(s-s_0)^2}\ds\dt
	- C_2 \int_{I\times (0,1)} |\psi|^2 \ds\dt.
\end{equation}

By (\ref{interpolation1}) and (\ref{interpolation3}), for each $\delta\in (0,1)$, one has
\begin{equation*}
	c_\varepsilon(\psi) - \left(\frac{\pi}{2\varepsilon}\right)^2\|\psi\|_{{\cal H}'_\varepsilon}^2
	\geq
	\frac{\delta}{16} \int_\Lambda \frac{|\psi|^2}{1+(s-s_0)^2}\ds\dt
	+ \left[(1-\delta)\mu_0^\varepsilon C_1 - \delta C_2\right] \int_{I\times (0,1)} |\psi|^2 \ds\dt.
\end{equation*}
Taking $\delta = \mu_0^\varepsilon C_1/(\mu_0^\varepsilon C_1+C_2)$, we have
\begin{align*}
c_\varepsilon(\psi) - \left(\frac{\pi}{2\varepsilon}\right)^2\|\psi\|_{{\cal H}'_\varepsilon}^2
& \geq 
\frac{\mu_0^\varepsilon C_1}{16(\mu_0^\varepsilon C_1+C_2)}
\int_\Lambda \frac{|\psi|^2}{1+(s-s_0)^2}\ds\dt \\
& \geq 
\frac{\mu_0^\varepsilon C_1}{16(\mu_0^\varepsilon C_1+C_2)} \inf_{s\in\mathbb{R}}\frac{1+s^2}{1+(s-s_0)^2}
\int_\Lambda \frac{1}{1+s^2} |\psi|^2 h_\varepsilon \ds \dt.
\end{align*}
The result follows with a real constant $c>0$ so that
$$c\leq \frac{\mu_0^\varepsilon C_1}{16(\mu_0^\varepsilon C_1+C_2)} \inf_{s\in\mathbb{R}}\frac{1+s^2}{1+(s-s_0)^2}.$$
\end{proof}

\begin{proof}[\bf Proof of Theorem \ref{hardy2}]
Define $C_3 := \varepsilon /(1-\varepsilon\|k\cdot\Theta\|_{L^\infty(\mathbb{R})})$. 
Since $|\Theta'(s)| \leq \delta/(1+s^2)$, 
we obtain
\begin{equation}\label{limitationmeasures}
\left|\frac{f_\varepsilon(s,t)}{h_\varepsilon(s,t)}-1\right|\leq  \frac{C_3 \delta}{1+s^2},	
\end{equation}
for all $(s,t)\in\Lambda$.
Recall that  $\rho(s)=1/(1+s^2)$ and assume $C_3 \delta <1$. For each $\psi\in \dom b_\varepsilon$, (\ref{limitationmeasures}) implies that
\begin{align*}
b_\varepsilon (\psi) - \left(\frac{\pi}{2\varepsilon}\right)^2\|\psi\|_{\cal H_\varepsilon}^2 
& \geq 
\frac{1}{1+C_3\delta} \int_\Lambda \frac{|\partial_s \psi|^2}{h_\varepsilon}\ds\dt 
+\int_\Lambda (1-C_3\delta \rho(s))\left(|\partial_t \psi|^2 - \left(\frac{\pi}{2\varepsilon}\right)^2|\psi|^2\right) h_\varepsilon \ds \dt \\
&
\quad \, - 2\left(\frac{\pi}{2\varepsilon}\right)^2 \int_{\Lambda}C_3\delta  \rho(s) |\psi|^2 h_\varepsilon \ds \dt\\
& \geq
\frac{1-C_3\delta}{1+C_3\delta}\left(c_\varepsilon(\psi)-\left(\frac{\pi}{2\varepsilon}\right)^2\|\psi\|_{{\cal H}'_\varepsilon}^2\right)
- 2\left(\frac{\pi}{2\varepsilon}\right)^2 \int_{\Lambda} C_3\delta \rho(s) |\psi|^2 h_\varepsilon \ds \dt.
\end{align*} 
By Theorem \ref{hardy1}, there exists a constant $c>0$ so that
\[b_\varepsilon(\psi) - \left(\frac{\pi}{2\varepsilon}\right)^2\|\psi\|_{\cal H_\varepsilon}^2 \geq 
\left(\frac{1-C_3\delta}{1+C_3\delta}\, c - C_3\delta  \frac{\pi^2}{2\varepsilon^2} \right) 
\int_{\Lambda} \rho(s) |\psi|^2 h_\varepsilon \ds \dt.\]
Now, just to take  $\delta>0$ small enough so that
\[\frac{1-C_3\delta}{1+C_3\delta}\, c - C_3\delta  \frac{\pi^2}{2\varepsilon^2}>0.\]
Thus, it follows the result.
\end{proof}

\section{Thin strips}\label{thinstrip}

In this section we present the proof of Theorem \ref{autovalores} stated in the Introduction. 
In fact, it is a consequence of Theorem \ref{theorem7} stated below. Before we introduce Theorem \ref{theorem7}, we present some important considerations.
Throughout this section, we assume the conditions in (\ref{mainassuption}) and (\ref{Thetalimited}); $\varepsilon$ is always small enough.

For each $s\in \mathbb{R}$, consider the auxiliary operator
\[T_\varepsilon(s) := -\partial_t^2 - \frac{\partial_t f_\varepsilon}{f_\varepsilon}  \partial_t, \qquad
\dom T_\varepsilon(s) = \{v \in H^2(0,1); v(0) = v'(1) = 0\},\]
acting in the Hilbert space $L^2((0,1),f_\varepsilon(s, \cdot) \dt)$.
Denote by $\Sigma(s,\varepsilon)$ the first eigenvalue of $T_\varepsilon(s)$.
By the analytic perturbation theory, 
\[\Sigma(s,\varepsilon) = \left(\frac{\pi}{2}\right)^2 + \beta(s,\varepsilon) + O(\varepsilon^2),\]
where
\begin{equation}\label{simplbeta}
	\beta(s,\varepsilon) :=  \frac{\varepsilon (k \cdot \Theta)(s) - \varepsilon^2 (k \cdot \Theta)^2(s) - \varepsilon^2 |\Theta'(s)|^2}{\sqrt{(1-\varepsilon (k\cdot\Theta)(s))^2 + \varepsilon^2 |\Theta(s)|^2}} + 
	\frac{1}{2}\int_0^1 (\partial^2_t f_\varepsilon)  \chi_1^2 \dt;
\end{equation}
%
see \cite{katotosio} for more details.

Recall the definition of $b_\varepsilon(\psi)$ in (\ref{formb}). For each $\psi \in \dom b_\varepsilon$, by the minimax principle,
\begin{equation}\label{blower}
	b_\varepsilon(\psi) 
	\geq  
	\int_\Lambda \frac{\Sigma(s,\varepsilon)}{\varepsilon^2}    
	|\psi|^2 f_\varepsilon \ds \dt.
\end{equation}
This estimate will be useful in the next paragraphs.
 
From now on, it is more convenient to perform  another change of coordinates so that the quadratic form $b_\varepsilon(\psi)$ starts to act in the Hilbert space $L^2(\Lambda)$ with the usual metric.  For this, consider the unitary operator
\[\begin{array}{rccc}
	{\cal V}_\varepsilon: &  L^2(\Lambda)  & \longrightarrow        & L^2(\Lambda, f_\varepsilon\ds\dt)\\
	& \psi    & \longmapsto &  f_\varepsilon^{-1/2} \psi 
\end{array},\]
and define
\[d_\varepsilon(\psi):=	b_\varepsilon \left({\cal V_\varepsilon} \psi \right),\]
$\dom d_\varepsilon := {\cal V}_\varepsilon^{-1}(\dom b_\varepsilon)$.
Due to the conditions in (\ref{Thetalimited}),  $\dom d_\varepsilon = {\cal D}:=\{ \psi\in H^1(\Lambda); \psi(s,0) =0 \textrm{ a.e. } s\in\mathbb{R}\}$.

Some calculations show that
\begin{equation*} 
	d_\varepsilon(\psi) =
	 \int_\Lambda \frac{|\partial_s \psi|^2}{ f_\varepsilon^2 } \ds \dt 
	+ \frac{1}{\varepsilon^2} \int_\Lambda |\partial_t \psi|^2 \ds\dt
	+\int_\Lambda V_\varepsilon |\psi|^2\ds\dt
	- {\cal R}\int_\Lambda \frac{\partial_s f_\varepsilon}{f_\varepsilon^3}\overline{\psi}\partial_s \psi \ds\dt 
	+ \int_{\partial} {\rm v}_\varepsilon |\psi|^2 \ds\dt,
\end{equation*}
where 
\[V_\varepsilon(s,t):= 
\frac{1}{4}\frac{(\partial_s f_\varepsilon(s,t))^2}{f_\varepsilon(s,t)^4}
-\frac{1}{4\varepsilon^2}\frac{(\partial_t f_\varepsilon(s,t))^2}{f_\varepsilon(s,t)^2} 
+\frac{1}{2\varepsilon^2}\frac{\partial_t^2 f_\varepsilon(s,t)}{f_\varepsilon(s,t)},\]
and
\[{\rm v}_\varepsilon(s):=
\frac{1}{2\varepsilon} \frac{(k\cdot\Theta)(s) - \varepsilon (k\cdot\Theta)(s)^2 - \varepsilon |\Theta'(s)|^2}{(1-\varepsilon (k\cdot\Theta)(s))^2 + \varepsilon^2 |\Theta(s)|^2}.
\]
The integral sign $\int_\partial$ refers to an integration over the
boundary $\mathbb{R}\times\{1\}$.
Denote by $D_\varepsilon$ the self-adjoint operator associated with  $d_\varepsilon(\psi)$.

In particular, by (\ref{simplbeta}) and (\ref{blower}), 
\begin{align}
	d_\varepsilon(\psi) 
	& 	\geq  	\left(\frac{\pi}{2\varepsilon}\right)^2 \int_\Lambda  |\psi|^2 \ds \dt
	+
	\int_\Lambda \frac{\beta(s,\varepsilon)}{\varepsilon^2}  	|\psi|^2 \ds \dt 	+ O(1) \int_\Lambda   	|\psi|^2 \ds \dt \nonumber \\ 
	& 	\geq  	\left(\frac{\pi}{2\varepsilon}\right)^2 \int_\Lambda  |\psi|^2 \ds \dt
	+
	\int_\Lambda \frac{k \cdot \Theta}{\varepsilon}  	|\psi|^2 \ds \dt 	+ O(1) \int_\Lambda   	|\psi|^2 \ds \dt. \label{dlower1}
\end{align}
As a consequence, one has
\begin{equation*}\label{Dlower}
	D_\varepsilon \geq \left[\left(\frac{\pi}{2\varepsilon} \right)^2 + \frac{k\cdot\Theta}{\varepsilon} + O(1)\right]\Id.
\end{equation*}
Let $H_\varepsilon$ be the decoupled operator 
\begin{equation*}\label{decoupledop}
H_\varepsilon:= \left(-\Delta_\mathbb{R}+\frac{k\cdot\Theta}{\varepsilon}\Id\right)\otimes \Id + \Id \otimes \left(-\frac{1}{\varepsilon^2}\Delta_{(0,1)}^{DN}\right)\quad \textrm{in}\quad L^2(\mathbb{R})\otimes L^2(0,1).
\end{equation*}
Fix a number $\boldsymbol{\kappa}>0$ so that
\[\boldsymbol{\kappa} + \inf k\cdot\Theta >0.\]
Consequently, $D_\varepsilon-(\pi/2\varepsilon)^2 \Id + (\boldsymbol{\kappa}/\varepsilon) \Id$ and 
$H_\varepsilon-(\pi/2\varepsilon)^2 \Id + (\boldsymbol{\kappa}/\varepsilon) \Id$ are positive operators.

Now, we have conditions to state the following result.

\begin{Theorem}\label{theorem7}
Assume the conditions {\rm (\ref{mainassuption})} and {\rm (\ref{Thetalimited})} in the Introduction. 
If $k\cdot\Theta\neq 0$, then there exist positive constants $\varepsilon_0$ and $K_0$ so that, for each $\varepsilon\in(0,\varepsilon_0)$,
\[\left\|\left[\left(D_{\varepsilon} -\left(\frac{\pi}{2\varepsilon}\right)^2\Id\right)  +\frac{\boldsymbol{\kappa}}{\varepsilon}\Id\right]^{-1} - \left[\left(H_\varepsilon-\left(\frac{\pi}{2\varepsilon}\right)^2\Id\right) +\frac{\boldsymbol{\kappa}}{\varepsilon}\Id\right]^{-1}\right\| \leq K_0\varepsilon^{3/2}.\]
\end{Theorem}

\medskip

The proof of Theorem \ref{theorem7} is divided into two steps;  Propositions \ref{proposition1} and \ref{proposition2} stated below.
For simplicity, we write 
\[L_\varepsilon :=\left(D_{\varepsilon} -\left(\frac{\pi}{2\varepsilon}\right)^2\Id\right)  +\frac{\boldsymbol{\kappa}}{\varepsilon}\Id.\]
The quadratic form associated with $L_\varepsilon$ is denoted by $l_\varepsilon(\psi)$;  $\dom l_\varepsilon = {\cal D}$.

Now, we introduce an intermediate operator in the following way.
The conditions in (\ref{mainassuption}) and (\ref{Thetalimited}) imply that
\begin{equation}\label{estimatemain}
	\|1/f_\varepsilon^2-1\|_{L^\infty(\Lambda)} \leq K_1 \varepsilon, 
	\qquad
	\left\|\partial_s f_\varepsilon/f_\varepsilon^3\right\|_{L^\infty(\Lambda)} \leq K_1 \varepsilon, 
	\qquad
	\| V_\varepsilon\|_{L^\infty(\Lambda)} \leq K_1, 
\end{equation}
for some $K_1> 0$.
These estimates motivate the definition of the quadratic form
\begin{equation*}
	m_\varepsilon (\psi) := \int_\Lambda |\partial_s \psi|^2 \ds \dt 
	+ \frac{1}{\varepsilon^2} \int_\Lambda \left[|\partial_t \psi|^2 - \left(\frac{\pi}{2}\right)^2 |\psi|^2 \right]\ds\dt 
	+ \int_{\Lambda}\frac{\boldsymbol{\kappa}}{\varepsilon} |\psi|^2 \ds\dt
	+ \int_{\partial} {\rm v}_\varepsilon |\psi|^2 \ds\dt,
\end{equation*}
$\dom m_\varepsilon = {\cal D}$. Denote by $M_\varepsilon$ the operator self-adjoint associated with $m_\varepsilon(\psi)$.

The first step is to prove that $L_\varepsilon$ approaches to $M_\varepsilon$ in the norm resolvent sense, as $\varepsilon$ goes to zero. 

\begin{Lemma}\label{lmlimited}
Assume the conditions {\rm (\ref{mainassuption})} and {\rm (\ref{Thetalimited})} in the Introduction.
Then there exists $K_2>0$ so that
\begin{equation*}
	l_\varepsilon(\psi) \geq K_2(\|\partial_s\psi\|_{L^2(\Lambda)}^2+\varepsilon^{-1}\|\psi\|_{L^2(\Lambda)}^2),
		\quad
	m_\varepsilon(\psi) \geq K_2(\|\partial_s\psi\|_{L^2(\Lambda)}^2+\varepsilon^{-1}\|\psi\|_{L^2(\Lambda)}^2),
\end{equation*}
for all $\psi \in {\cal D}$, for all $\varepsilon>0$ small enough.
\end{Lemma}
\begin{proof}
By (\ref{dlower1}), it is easy to note that
\begin{equation}\label{llower}
	l_\varepsilon(\psi) \geq K_3 (\|\partial_s\psi\|_{L^2(\Lambda)}^2+\varepsilon^{-1}\|\psi\|_{L^2(\Lambda)}^2),
\end{equation}
for some $K_3>0$.
By (\ref{estimatemain}),
\begin{align}\label{lmenosm}
	|l_\varepsilon(\psi) - m_\varepsilon(\psi)| 
	&\leq 
	\int_\Lambda \left|\frac{1}{f_\varepsilon^2} - 1\right| |\partial_s \psi|^2 \ds \dt  
	+\int_\Lambda |V_\varepsilon| |\psi|^2\ds\dt
	+ \int_\Lambda \left|\frac{\partial_s f_\varepsilon}{f_\varepsilon^3}\right| |\psi||\partial_s \psi| \ds\dt\nonumber\\
	& \leq 
	K_4(\varepsilon\|\partial_s \psi\|_{L^2(\Lambda)}^2+ \|\psi\|_{L^2(\Lambda)}^2),
\end{align}
for some $K_4>0$.
Then, due to (\ref{llower}) and (\ref{lmenosm}), 
\begin{equation*}
	m_\varepsilon(\psi) \geq K_5(\|\partial_s\psi\|_{L^2(\Lambda)}^2+\varepsilon^{-1}\|\psi\|_{L^2(\Lambda)}^2),
\end{equation*}
for some $K_5>0$.
\end{proof}

As a consequence of Lemma \ref{lmlimited}, $L_\varepsilon$ and $M_\varepsilon$ are positive operators and
\begin{equation}\label{operatorlimited2}
	\|L_\varepsilon^{-1}\|\leq K_2^{-1}\varepsilon, 
	\qquad \qquad
	\|M_\varepsilon^{-1}\|\leq K_2^{-1}\varepsilon.
\end{equation}
These estimates will be useful latter.

\begin{Proposition}\label{proposition1}
Assume the conditions {\rm (\ref{mainassuption})} and {\rm (\ref{Thetalimited})} in the Introduction. 
Then there exist positive constants $\varepsilon_1$ and $K_6$ so that, for each $\varepsilon\in(0,\varepsilon_1)$,
\[\|L_\varepsilon^{-1}-M_\varepsilon^{-1}\|\leq K_6\varepsilon^2.\]
\end{Proposition}
\begin{proof}
Denote by $l_\varepsilon(\phi,\varphi)$ and $m_\varepsilon(\phi,\varphi)$ the sesquilinear forms associated with $l_\varepsilon(\psi)$ and $m_\varepsilon(\psi)$, respectively. 
Due to the estimates in (\ref{estimatemain}), 
\begin{align*}
|l_\varepsilon(\phi,\varphi)-m_\varepsilon(\phi,\varphi)| 
&\leq 
\int_\Lambda \left|\frac{1}{f_\varepsilon^2}-1\right||\partial_s\phi| |\partial_s\varphi|\ds\dt
+ \int_\Lambda |V_\varepsilon| |\phi| |\varphi| \ds\dt
+ \int_\Lambda \left|\frac{\partial_s f_\varepsilon}{f_\varepsilon^3}\right| |\overline{\phi}\partial_s \varphi + \overline{\varphi}\partial_s \phi| \ds\dt\\ 
& \leq 
K_7 \sqrt{\varepsilon\|\partial_s \phi\|_{L^2(\Lambda)}^2+ \|\phi\|_{L^2(\Lambda)}^2} \sqrt{\varepsilon\|\partial_s \varphi\|_{L^2(\Lambda)}^2+ \|\varphi\|_{L^2(\Lambda)}^2},
\end{align*}
for some $K_7 > 0$.
By Lemma \ref{lmlimited}, 
\begin{equation}\label{lmenosmsesquilinear}
	|l_\varepsilon(\phi,\varphi)-m_\varepsilon(\phi,\varphi)| 
	\leq 
	K_8 \varepsilon \sqrt{m_\varepsilon(\phi) l_\varepsilon(\varphi)},
\end{equation}
for some $K_8 > 0$.

Now, let $\eta, \xi\in L^2(\Lambda)$ be arbitrary functions. Take $\phi = M_\varepsilon^{-1}\eta$ and $\varphi =L_\varepsilon^{-1}\xi$. 
Denoting by  $\langle\cdot,\cdot\rangle$ the inner product in $L^2(\Lambda)$, we have 
$\langle M_\varepsilon^{-1}\eta, \eta \rangle = m_\varepsilon(\phi)$, $\langle L_\varepsilon^{-1}\xi, \xi \rangle = l_\varepsilon(\varphi)$, and
\[|\langle (L_\varepsilon^{-1}-M_\varepsilon^{-1})\eta, \xi \rangle| = | m_\varepsilon(\phi,\varphi)-l_\varepsilon(\phi,\varphi)|.\]
Consequently, by the estimate in (\ref{lmenosmsesquilinear}), 
\begin{equation}\label{LmenosM}
\left|\langle (L_\varepsilon^{-1}-M_\varepsilon^{-1})\eta, \xi \rangle\right| \leq 
K_8\varepsilon \sqrt{\langle M_\varepsilon^{-1}\eta, \eta \rangle\langle L_\varepsilon^{-1}\xi, \xi \rangle}.
\end{equation}
Finally, 
(\ref{operatorlimited2}) and (\ref{LmenosM}) imply that
\begin{equation*}
	\left|\langle (L_\varepsilon^{-1}-M_\varepsilon^{-1})\eta, \xi \rangle\right| \leq K_9\varepsilon^2 \|\eta\|_{L^2(\Lambda)} \|\xi\|_{L^2(\Lambda)},
\end{equation*}
for some $K_9>0$.
This completes the proof.
\end{proof}

The next step is to prove that $M_\varepsilon$ approaches to the self-adjoint operator 
\[N_\varepsilon := \left(H_\varepsilon-\left(\frac{\pi}{2\varepsilon}\right)^2\Id\right) +\frac{\boldsymbol{\kappa}}{\varepsilon}\Id\]
in the norm resolvent sense, as $\varepsilon$ goes to zero.
Denote by $n_\varepsilon(\psi)$ the quadratic forms associated with $N_\varepsilon$; $\dom n_\varepsilon =  {\cal D}$.
A straightforward calculation shows that
\begin{equation*}
	n_\varepsilon(\psi) = \int_\Lambda |\partial_s \psi|^2 \ds \dt 
	+ \frac{1}{\varepsilon^2} \int_\Lambda \left[|\partial_t \psi|^2 - \left(\frac{\pi}{2}\right)^2 |\psi|^2 \right]\ds\dt 
	+ \int_{\Lambda}\frac{k\cdot\Theta +\boldsymbol{\kappa}}{\varepsilon} |\psi|^2 \ds\dt.
\end{equation*}
We can see that 
\begin{equation}\label{nlimited}
	n_\varepsilon(\psi) \geq K_{10}(\|\partial_s\psi\|_{L^2(\Lambda)}^2+\varepsilon^{-1}\|\psi\|_{L^2(\Lambda)}^2),
\end{equation}
for some $K_{10}>0$. Consequently, $N_\varepsilon$ is a positive operator and
\begin{equation}\label{Nlimited}
	\|N_\varepsilon^{-1}\|\leq K_{10}^{-1}\varepsilon.
\end{equation}

\begin{Proposition}\label{proposition2}
	Assume the conditions {\rm (\ref{mainassuption})} and {\rm (\ref{Thetalimited})} in the Introduction. 
	If $k\cdot\Theta\neq 0$, then there exist positive constants $\varepsilon_2$ and $K_{11}$ so that, for each $\varepsilon\in(0,\varepsilon_2)$,
	\[\|M_\varepsilon^{-1}-N_\varepsilon^{-1}\|\leq K_{11}\varepsilon^{3/2}.\]
\end{Proposition}
\begin{proof}
Consider the closed subspace 
\[{\cal A}:= \{\psi_w:= w(s)\chi_1(t): w \in H^1(\mathbb R)\}\]
of the Hilbert space $L^2(\Lambda)$;
recall that $\chi_1$ is given by (\ref{chidefinition}).
One has the orthogonal decomposition $L^2(\Lambda)= {\cal A}\oplus{\cal A}^\perp$. 
Each $\psi\in {\cal D}$ can be written as
\begin{equation}\label{decperpproo}
\psi = \psi_w + \psi_w^\perp, \quad w \in H^1(\mathbb R), \psi_w^\perp \in {\cal D} \cap {\cal A}^\perp.
\end{equation}
In particular, for each $\psi_w^\perp\in {\cal Q} \cap {\cal A}^\perp$,
\begin{equation}\label{eq31}
	\int_0^1\psi_w^\perp(s,t)\chi_1(t)\dt =0 \quad 
	\hbox{and} \quad
	\int_0^1\partial_s\psi_w^\perp(s,t)\chi_1(t)\dt =0,
\end{equation}
for almost every $s\in\mathbb{R}$.

According to (\ref{decperpproo}),
\[n_\varepsilon(\psi) = n_\varepsilon(\psi_w) + 2 {\cal R} n_\varepsilon(\psi_w, \psi_w^\perp) + n_\varepsilon(\psi_w^\perp);\]
$n_\varepsilon(\varphi, \phi)$ denotes the sesquilinear form associated with $n_\varepsilon(\psi)$.
As a consequence of (\ref{eq31}),   $n_\varepsilon(\psi_w,\psi_w^\perp)=0$.
We also have
\begin{align}
n_\varepsilon(\psi_w^\perp) 
&\geq 
\frac{1}{\varepsilon^2} \|\partial_t \psi_w^\perp\|_{L^2(\Lambda)}^2 - \left(\frac{\pi}{2\varepsilon}\right)^2 \|\psi_w^\perp\|_{L^2(\Lambda)}^2 \nonumber \\
&= 
\frac{1}{2\varepsilon^2} \|\partial_t \psi_w^\perp\|_{L^2(\Lambda)}^2 +\frac{1}{2\varepsilon^2} \|\partial_t \psi_w^\perp\|_{L^2(\Lambda)}^2 - \left(\frac{\pi}{2\varepsilon}\right)^2 \|\psi_w^\perp\|_{L^2(\Lambda)}^2 \nonumber \\
&\geq 
\frac{1}{2\varepsilon^2} \|\partial_t \psi_w^\perp\|_{L^2(\Lambda)}^2 
+ \left(\frac{1}{2}\left(\frac{3\pi}{2\varepsilon}\right)^2 - \left(\frac{\pi}{2\varepsilon}\right)^2\right)  \|\psi_w^\perp\|_{L^2(\Lambda)}^2 \nonumber \\
& \geq
\frac{1}{2 \varepsilon^2} \left(\|\partial_t \psi_w^\perp\|_{L^2(\Lambda)}^2 
+ \|\psi_w^\perp\|_{L^2(\Lambda)}^2\right). \label{eq33}
\end{align}

Now, we define the auxiliary quadratic form
\begin{align*}
q_\varepsilon(\psi) :=  m_\varepsilon(\psi)-n_\varepsilon(\psi)
= 
\int_\partial {\rm v}_\varepsilon|\psi|^2\ds\dt
-\int_\Lambda \frac{k\cdot\Theta}{\varepsilon}|\psi|^2\ds\dt,
\end{align*}
$\dom q_\varepsilon = {\cal D}$.
In particular,
\begin{equation}\label{formq}
	|q_\varepsilon(\psi)| 
	\leq 
	|q_\varepsilon(\psi_w)| + 2|q_\varepsilon(\psi_w,\psi_w^\perp)| + |q_\varepsilon(\psi_w^\perp)|;
\end{equation}
$q_\varepsilon(\phi,\varphi)$ denotes the sesquilinear form associated with $q_\varepsilon(\psi)$.
A straightforward calculation shows that
\[q_\varepsilon(\psi_w) = \int_\mathbb{R}\left(2 {\rm v}_\varepsilon  
- \frac{k\cdot\Theta}{\varepsilon}\right)|w|^2\ds.\]
Since $k\cdot\Theta \neq 0$, (\ref{mainassuption}) and (\ref{Thetalimited}) imply that
\[\|2{\rm v}_\varepsilon - k\cdot\Theta/\varepsilon\|_{L^\infty(\mathbb{R})} \leq K_{12},\qquad 
\| {\rm v}_\varepsilon\|_{L^\infty(\mathbb{R})} \leq K_{12} \varepsilon^{-1},\]
for some $K_{12}>0$.
These estimates, together with (\ref{nlimited}) and (\ref{eq33}), ensure that there exist $K_{13}, K_{14}, K_{15} > 0$ so that
\begin{equation}
	|q_\varepsilon(\psi_w)| 
	\leq K_{12} \|w\|_{L^2(\mathbb{R})}^2
	\leq K_{13} \varepsilon n_\varepsilon(\psi_w),
\end{equation}
and
\begin{align}
	|q_\varepsilon(\psi_w^\perp)|  
	& \leq 
	K_{12} \varepsilon^{-1} \left(\int_\partial |\psi_w^\perp|^2\ds\dt + \int_\Lambda |\psi_w^\perp|^2\ds\dt\right) \nonumber \\
	& = K_{12} \varepsilon^{-1} \left( 2 {\cal R} \int_\Lambda \overline{\psi_w^\perp} \partial_t \psi_w^\perp \ds\dt + \int_\Lambda |\psi_w^\perp|^2 \ds\dt\right) \nonumber \\
	&\leq 	
	K_{12} \varepsilon^{-1} \left( 2 \|\psi_w^\perp\|_{L^2(\Lambda)} \|\partial_t \psi_w^\perp\|_{L^2(\Lambda)} 
	+ \|\psi_w^\perp\|_{L^2(\Lambda)}^2\right) \nonumber\\
	&\leq 
	K_{14} \varepsilon^{-1} \left(\|\partial_t \psi_w^\perp\|_{L^2(\Lambda)}^2 
	+ \|\psi_w^\perp\|_{L^2(\Lambda)}^2\right) \nonumber \\
	& \leq K_{15}\varepsilon n_\varepsilon(\psi_w^\perp).\label{secondpart}
\end{align}

Finally, some calculations show that
\begin{align*}
|q_\varepsilon(\psi_w, \psi_w^\perp)| = \left| \int_\partial {\rm v}_\varepsilon \psi_w\overline{\psi_w^\perp} \ds\dt \right|
&\leq 
K_{12}\varepsilon^{-1} \left(\int_\partial |\psi_w|^2\ds\dt\right)^{1/2} \left(\int_\partial |\psi_w^\perp|^2\ds\dt\right)^{1/2}\\
&\leq 
2 K_{12} \varepsilon^{-1} \|w\|_{L^2(\mathbb{R})}
\left(\|\psi_w^\perp\|_{L^2(\Lambda)}\|\partial_t \psi_w^\perp\|_{L^2(\Lambda)}\right)^{1/2} \\
&\leq 
2 K_{12} \varepsilon^{-1} \|w\|_{L^2(\mathbb{R})}
\left(\|\partial_t \psi_w^\perp\|_{L^2(\Lambda)}^2 + \|\psi_w^\perp\|_{L^2(\Lambda)}^2\right)^{1/2}.
\end{align*}
Combining these estimates with (\ref{nlimited}) and (\ref{eq33}), 
\begin{align}\label{thirdpart}
	|q_\varepsilon(\psi_w, \psi_w^\perp)| 
	&\leq 
	K_{16} \varepsilon^{1/2} \sqrt{n_\varepsilon(\psi_w) n_\varepsilon(\psi_w^\perp)}\nonumber \\
	&\leq 
	K_{16}\varepsilon^{1/2} 
	\left(n_\varepsilon(\psi_w) +  n_\varepsilon(\psi_w^\perp)\right),
\end{align}
for some $K_{16} > 0$.
As a consequence of (\ref{formq})-(\ref{thirdpart}), 
\begin{equation}\label{qlimited}
|q_\varepsilon(\psi)| 
\leq 
K_{17} \varepsilon^{1/2} (n_\varepsilon(\psi_w)+n_\varepsilon(\psi_w^\perp))
= K_{17} \varepsilon^{1/2} n_\varepsilon(\psi),
\end{equation}
for some $K_{17} > 0$.

By the polarization identity and  (\ref{qlimited}), 
\begin{equation}\label{qsesquilinear}
|q_\varepsilon(\phi,\varphi)| 
\leq 
K_{17}\varepsilon^{1/2}\left(n_\varepsilon(\phi) + n_\varepsilon(\varphi)\right).
\end{equation}
Furthermore, (\ref{qlimited}) implies that
\begin{equation}\label{eq35}
	K_{18}n_\varepsilon(\varphi)\leq m_\varepsilon(\varphi) \leq K_{18} n_\varepsilon(\varphi),
\end{equation}
for some $K_{18}> 0$.
Then, as consequence of (\ref{qsesquilinear}) and (\ref{eq35}), 
\begin{equation}\label{estimateq}
	|q_\varepsilon(\phi,\varphi)| \leq 
	K_{19} \varepsilon^{1/2}\left(n_\varepsilon(\phi) + m_\varepsilon(\varphi)\right),
\end{equation}
for some $K_{19}>0$.

Now, let $\eta, \xi\in L^2(\Lambda)$ be arbitrary functions.
Similarly as in the proof of Proposition \ref{proposition1}, it follows from (\ref{estimateq}) that
\begin{equation}\label{MmenosN}
\left|\langle (M_\varepsilon^{-1}-N_\varepsilon^{-1})\eta, \xi \rangle\right| \leq 
K_{19} \varepsilon^{1/2} \left(\langle N_\varepsilon^{-1}\eta, \eta \rangle
+ \langle M_\varepsilon^{-1}\xi, \xi \rangle\right).
\end{equation}
Using (\ref{MmenosN}) together with (\ref{operatorlimited2}) and (\ref{Nlimited}), one has 
\[\left|\langle (M_\varepsilon^{-1}-N_\varepsilon^{-1})\eta, \xi \rangle\right| \leq 
K_{20}\varepsilon^{3/2} \left(\|\eta\|_{L^2(\Lambda)} + \|\xi\|_{L^2(\Lambda)}\right),\]
for some $K_{20}>0$. Thus, we get the result.
\end{proof}

\begin{proof}[\bf Proof of Theorem \ref{theorem7}]
Apply Propositions \ref{proposition1} and \ref{proposition2}.
\end{proof}

Finally, we present the proof of Theorem \ref{autovalores}.
We start with the following result.

\begin{Lemma}\label{lemma3}
	Suppose that $k\cdot\Theta\in L^\infty(\mathbb{R})$. One has
	\[\lambda_1(H_\varepsilon) = \lambda_1\left(-\Delta_\mathbb{R}+\frac{k\cdot\Theta}{\varepsilon}\Id\right)+\left(\frac{\pi}{2\varepsilon}\right)^2.\]
	Furthermore, for any integer $N\geq2$, there exists a positive constant $\varepsilon_3 := \varepsilon_3(N,k\cdot\Theta)$ so that, for each $\varepsilon\in(0,\varepsilon_3)$,
	\[\lambda_j(H_\varepsilon) = \lambda_j\left(-\Delta_\mathbb{R}+\frac{k\cdot\Theta}{\varepsilon}\Id\right)+\left(\frac{\pi}{2\varepsilon}\right)^2,\quad \forall j\in\{2,\cdots,N\}. \]
\end{Lemma}

The proof of this lemma can be obtained by simple adaptations of the proof of Lemma 3.5 of \cite{daviddn},
which will be omitted in this text.

\begin{proof}[\bf Proof of Theorem \ref{autovalores}]
For each $j\in\mathbb{N}$, consider $\varepsilon>0$ small enough so that the conclusions of Theorem \ref{theorem7} and Lemma \ref{lemma3} hold. 
%
%
%
%
%
Theorem \ref{theorem7} and Lemma \ref{lemma3} ensures that
\[\left| \frac{1}{\varepsilon (\lambda_j(-\Delta_{\varepsilon}^{DN}) - (\pi/2\varepsilon)^2)  + \boldsymbol{\kappa}}
- \frac{1}{\varepsilon \lambda_j(-\Delta_{\mathbb{R}} -(k\cdot\Theta/\varepsilon)\Id)  + \boldsymbol{\kappa}}\right|\leq K_0 \varepsilon^{1/2}.\]
By Theorem \ref{limiteautovalores} in Appendix \ref{appendix001}, one has
\[\lim_{\varepsilon\to 0} \varepsilon \left[\lambda_j(-\Delta_{\varepsilon}^{DN})  - \left(\frac{\pi}{2\varepsilon}\right)^2\right]
=
\lim_{\varepsilon\to 0} \varepsilon \lambda_j \left(-\Delta_\mathbb R  + \frac{k\cdot\Theta}{\varepsilon}\Id \right)=
\inf (k\cdot\Theta).\]
\end{proof}

\subsection*{Purely twisted strips}

We finish this section with the proof of Theorem \ref{asymtwisted}. 
In particular, it is a consequence of the convergence in the norm resolvent sense of the family of operators $\{D_\varepsilon\}_{\varepsilon>0}$.
Hereafter, we assume that $k\cdot \Theta =0$. 
Under this condition, we keep the same notation for the auxiliary operators and their respective quadratic forms.

From now on, for technical reasons, we take
\[\boldsymbol{\kappa} > \frac{\sup |\Theta'|^2}{2}.\]
Since $k\cdot\Theta = 0$ and $\Theta$ satisfies (\ref{Thetalimited}), one has
\begin{equation*}\label{estimate2}
	\|1/f_\varepsilon^2-1\|_{L^\infty(\Lambda)} \leq  K_{21} \varepsilon, 
	\qquad
	\|\partial_s f_\varepsilon/f_\varepsilon^3\|_{L^\infty(\Lambda)} \leq  K_{21} \varepsilon, 
	\qquad
	\| V_\varepsilon - {|\Theta'|^2}/{2}\|_{L^\infty(\Lambda)} \leq K_{21} \varepsilon, 
\end{equation*}
for some $K_{21}>0$.
These estimates motivate the definition of the quadratic form
\begin{align*}
	m_\varepsilon(\psi) := & 
	\int_\Lambda |\partial_s \psi|^2 \ds \dt
	+ \int_\Lambda \left( \frac{|\partial_t \psi|^2}{\varepsilon^2} - 
	\left(\frac{\pi}{2\varepsilon}\right)^2 |\psi|^2 \right)\ds \dt
	+ \int_\Lambda \frac{|\Theta'|^2}{2} |\psi|^2  \ds \dt \\
	& + 	\int_\partial {\rm v}_\varepsilon |\psi|^2 \ds \dt
	+ \boldsymbol{\kappa} \int_\Lambda |\psi|^2 \ds \dt,
\end{align*}
$\dom {m}_\varepsilon = {\cal D}$.
Denote by ${M}_\varepsilon$ the self-adjoint operator
associated with ${m}_\varepsilon(\psi)$.
Theorem 3 of \cite{oliveira01} ensures that
\begin{equation}\label{normresolvent1}
\left\|\left[\left(D_{\varepsilon} -\left(\frac{\pi}{2\varepsilon}\right)^2\Id\right)  +\boldsymbol{\kappa}\Id\right]^{-1} - M_\varepsilon^{-1}\right\| \leq K_{22} \varepsilon,
\end{equation}
for some ${K}_{22}>0$.

Now, define the one-dimensional quadratic form 
\begin{equation*}
	{n}_\varepsilon(w) := {m}_\varepsilon(w\chi_1) = 
	\int_\mathbb{R}\left(|w'|^2 
	+ \left(\frac{|\Theta'|^2}{2} + 2 {\rm v}_\varepsilon \right) |w|^2
	\right) \ds \dt
	+ \boldsymbol{{\kappa}} \int_\mathbb{R} |w|^2 \ds \dt,
\end{equation*} 
$\dom {n}_\varepsilon = H^1(\mathbb R)$.
Denote by ${N}_\varepsilon$ the self-adjoint operator
associated with ${n}_\varepsilon(\psi)$.
Considering the orthogonal decomposition given in (\ref{decperpproo}), we get the following result.

\begin{Proposition}\label{propalmfin}
	There exists a number ${K}_{23}>0$ so that
	\begin{equation}
		\| {M}_\varepsilon^{-1} - {N}_\varepsilon^{-1}\oplus{\bf 0} \| \leq {K}_{23} \varepsilon^{1/2},
	\end{equation}
	for all $\varepsilon>0$ small enough, where ${\bf 0}$ is the null operator on the subspace ${\cal A}^\perp$.
\end{Proposition}
\begin{proof}
	For each $\psi\in {{\cal D}}$, it follows from (\ref{decperpproo}) that
	\begin{equation*}
		\psi = w\chi_1+\psi_w^\perp, \quad w\in H^1(\mathbb{R}),  \psi_w^\perp\in {{\cal D}} \cap {\cal A}^\perp.
	\end{equation*}
	Note that
	\begin{equation*}
		{m}_\varepsilon(w\chi_1) = {n}_\varepsilon(w) \geq {K}_{24} \|w\|_{L^2(\mathbb{R})}^2,	
	\end{equation*}
	for some ${K}_{24} > 0$, and
	\begin{equation*}
		{m}_\varepsilon(\psi_w^\perp) 
		\geq
		\frac{1}{2\varepsilon^2} \left(\|\partial_t \psi_w^\perp\|_{L^2(\Lambda)}^2 
		+  \|\psi_w^\perp\|_{L^2(\Lambda)}^2\right);
	\end{equation*}
	the last estimate  follows the lines of (\ref{eq33}) in Section \ref{thinstrip}.
	Denote by ${m}_\varepsilon(\phi,\varphi)$ the sesquilinear form associated with ${m}_\varepsilon(\psi)$.
	Following the same procedure as the first inequality of (\ref{thirdpart}), 	
	\[{m}_\varepsilon(w\chi_1,\psi_w^\perp) \leq {K}_{25} \varepsilon^{1/2} \sqrt{{n}_\varepsilon(w) {m}_\varepsilon(\psi_w^\perp)},\]
	for some ${K}_{25} > 0$. Then, the result follows by Proposition 3.1 of \cite{solomyak}.
\end{proof}

Since $k\cdot\Theta = 0$ and $|\Theta'|\in L^\infty(\mathbb{R})$, one has
\begin{equation*}
	\| 2{\rm v}_\varepsilon  + |\Theta'|^2\|_{L^\infty(\mathbb{R})} \leq K_{26} \varepsilon, 
\end{equation*}
for some $K_{26} > 0$.
Thus, define the quadratic form
\begin{equation*}
	n(w) :=
	\int_\mathbb{R}\left(|w'|^2 
	 - \frac{|\Theta'|^2}{2} |w|^2
	\right) \ds \dt,
	+ \boldsymbol{\kappa} \int_\mathbb{R} |w|^2 \ds \dt,
\end{equation*}
$\dom n = H^1(\mathbb{R})$.
Denote by $N$ the self-adjoint operator
associated with $n(w)$. Again, 
Theorem 3 of \cite{oliveira01} implies that
\begin{equation}\label{normresolvent2}
	\|{N}_\varepsilon^{-1} - N^{-1} \| \leq K_{27} \varepsilon,
\end{equation}
for some $K_{27}>0$.

As a consequence of (\ref{normresolvent1})-(\ref{normresolvent2}), 
%
\begin{equation}\label{normresolvent}
\left\|\left[\left(D_{\varepsilon} -\left(\frac{\pi}{2\varepsilon}\right)^2\Id\right)  +\boldsymbol{\kappa}\Id\right]^{-1} - \left[\left(- \Delta_{\mathbb{R}}-\frac{|\Theta'|^2}{2}\Id\right) +\boldsymbol{\kappa}\Id\right]^{-1}\oplus{\bf 0}\right\| \leq K_{28}\varepsilon^{1/2},
\end{equation}
for some $K_{28}>0$.

\begin{proof}[\bf Proof of Theorem \ref{asymtwisted}]
Follows from estimate (\ref{normresolvent}).
\end{proof}

\section{Scaled strips}\label{twistedeffect}

Based on arguments of Section \ref{thinstrip}, we prove that  $\varepsilon(-\tilde{\Delta}_\varepsilon^{DN} - (\pi/2\varepsilon)^2\Id)$
converges in the norm resolvent sense, as $\varepsilon$ goes to zero. 
Recall that $-\tilde{\Delta}_\varepsilon^{DN}$ is the self-adjoint operator associated with the quadratic form $\tilde{a}_\varepsilon(\varphi)$; see (\ref{qtilde}).
Consequently, one proves Theorem \ref{lambdaassintotico} in the Introduction.
In this section, we also assume the conditions in (\ref{mainassuption}) and (\ref{Thetalimited}); $\varepsilon$ is always small enough.

At first, define
\[\tilde{f}_\varepsilon(s,t):=\sqrt{(1 - \varepsilon t (k_\varepsilon\cdot \Theta_\varepsilon)(s))^2 +  \varepsilon^2 t^2 |\Theta_\varepsilon'(s)|^2}.\]
Performing a change of coordinates similar to that in Section \ref{thinstrip}, 
$\tilde{a}_\varepsilon(\varphi)$ becomes
\begin{equation*} 
	\tilde{d}_\varepsilon(\psi) :=
	\int_\Lambda \frac{|\partial_s \psi|^2}{ \tilde{f}_\varepsilon^2 } \ds \dt 
	+ \frac{1}{\varepsilon^2} \int_\Lambda |\partial_t \psi|^2 \ds\dt
	+\int_\Lambda \widetilde{V}_\varepsilon |\psi|^2\ds\dt
	- {\cal R}\int_\Lambda \frac{\partial_s \tilde{f}_\varepsilon}{\tilde{f}_\varepsilon^3}\overline{\psi}\partial_s \psi \ds\dt 
	+ \int_{\partial} \tilde{\rm v}_\varepsilon |\psi|^2 \ds\dt,
\end{equation*}
$\dom \tilde{d}_\varepsilon := \{ \psi\in H^1(\Lambda); \psi(s,0) =0 \textrm{ a.e. } s\in\mathbb{R}\}\subset L^2(\Lambda)$, where %
\[\widetilde{V}_\varepsilon(s,t):= 
\frac{1}{4}\frac{(\partial_s \tilde{f}_\varepsilon(s,t))^2}{\tilde{f}_\varepsilon(s,t)^4}
-\frac{1}{4\varepsilon^2}\frac{(\partial_t \tilde{f}_\varepsilon(s,t))^2}{\tilde{f}_\varepsilon(s,t)^2} 
+\frac{1}{2\varepsilon^2}\frac{\partial_t^2 \tilde{f}_\varepsilon(s,t)}{\tilde{f}_\varepsilon(s,t)},\]
and
\[\widetilde{\rm v}_\varepsilon(s):=
\frac{1}{2\varepsilon} \frac{(k_\varepsilon\cdot\Theta_\varepsilon)(s) - \varepsilon (k_\varepsilon\cdot\Theta_\varepsilon)(s)^2 - \varepsilon |\Theta_\varepsilon'(s)|^2}{(1-\varepsilon (k_\varepsilon\cdot\Theta_\varepsilon)(s))^2 + \varepsilon^2 |\Theta_\varepsilon'(s)|^2}.
\]
Recall that the integral sign $\int_\partial$ refers to an integration over the
boundary $\mathbb{R}\times\{1\}$.

Now, we perform another change of coordinates in order to work with the functions 
$k\cdot\Theta$ and $|\Theta'|$
instead of $k_\varepsilon\cdot\Theta_\varepsilon$ and $|\Theta_\varepsilon'|$. 
Let 
${\cal W}_\varepsilon : L^2(\Lambda) \longrightarrow L^2(\Lambda)$
be the unitary operator 
generated by the horizontal dilation $(s,t) \longmapsto (\varepsilon^{1/2} s, t)$, i.e., 
\[({\cal W}_\varepsilon \psi)(s,t) := \varepsilon^{1/4} \psi (\varepsilon^{1/2} s, t).\]
Then, define the quadratic form 
\begin{equation*}\label{formc} 
	y_\varepsilon(\psi)  := \tilde{d}_\varepsilon \left({\cal W}_\varepsilon^{-1} \psi \right),
\end{equation*}
$\dom y_\varepsilon := \{ \psi\in H^1(\Lambda); \psi(s,0) =0 \textrm{ a.e. } s\in\mathbb{R}\}$, and  denote 
\[\tilde{h}_\varepsilon(s,t):=\sqrt{(1 - \varepsilon t (k\cdot \Theta)(s))^2 +  \varepsilon t^2 |\Theta'(s)|^2}.\]
A straightforward calculation shows that
\begin{equation*} 
	y_\varepsilon(\psi) :=
	\frac{1}{\varepsilon} \left[
	\int_\Lambda \frac{|\partial_s \psi|^2}{\tilde{h}_\varepsilon^2 } \ds \dt 
	+ \frac{1}{\varepsilon} \int_\Lambda |\partial_t \psi|^2 \ds\dt
	+\int_\Lambda W_\varepsilon |\psi|^2\ds\dt
	- {\cal R}\int_\Lambda \frac{\partial_s \tilde{h}_\varepsilon}{\tilde{h}_\varepsilon^3}\overline{\psi}\partial_s \psi \ds\dt 
	+ \int_{\partial} {\rm w}_\varepsilon |\psi|^2 \ds\dt \right],
\end{equation*}
where
\[W_\varepsilon(s,t):= 
\frac{1}{4}\frac{(\partial_s \tilde{h}_\varepsilon(s,t))^2}{\tilde{h}_\varepsilon(s,t)^4}
-\frac{1}{4\varepsilon}\frac{(\partial_t \tilde{h}_\varepsilon(s,t))^2}{\tilde{h}_\varepsilon(s,t)^2} 
+\frac{1}{2\varepsilon}\frac{\partial_t^2 \tilde{h}_\varepsilon(s,t)}{\tilde{h}_\varepsilon(s,t)},\]
and
\[{\rm w}_\varepsilon(s):=
\frac{1}{2} \frac{(k\cdot\Theta)(s) - \varepsilon (k\cdot\Theta)(s)^2 - |\Theta'(s)|^2}{(1-\varepsilon (k\cdot\Theta)(s))^2 + \varepsilon |\Theta'(s)|^2}.\]

For technical reasons, take a constant 
\[\boldsymbol{\tilde{\kappa}} > \|(k \cdot \Theta) - |\Theta'|^2/2\|_\infty,\]
and define 
$$\tilde{l}_\varepsilon(\psi) := 
\varepsilon \left[y_\varepsilon(\psi) - \left(\frac{\pi}{2\varepsilon}\right)^2 \int_\Lambda |\psi|^2 \ds \dt\right] 
+ \boldsymbol{\tilde{\kappa}} \int_\Lambda |\psi|^2 \ds \dt,$$ 
$\dom \tilde{l}_\varepsilon =  \tilde{\cal D} := \{ \psi \in H^1(\Lambda); \psi(s,0) =0 \textrm{ a.e. } s\in\mathbb{R}\}$.
Denote by $\widetilde{L}_\varepsilon$ the self-adjoint operator
associated with  $\tilde{l}_\varepsilon(\psi)$.
Due to conditions (\ref{mainassuption}) and (\ref{Thetalimited}), one has
\begin{equation}\label{label01}
\|1/\tilde{h}_\varepsilon^2-1\|_{L^\infty(\Lambda)} \leq  \tilde{K}_1 \varepsilon, 
\qquad
\|\partial_s \tilde{h}_\varepsilon/\tilde{h}_\varepsilon^3\|_{L^\infty(\Lambda)} \leq  \tilde{K}_1 \varepsilon,
\end{equation}
\begin{equation}\label{label02}
\| W_\varepsilon - {|\Theta'|^2}/{2}\|_{L^\infty(\Lambda)} \leq \tilde{K}_1 \varepsilon, 
\qquad
\| 2{\rm w}_\varepsilon - k\cdot\Theta + |\Theta'|^2\|_{L^\infty(\mathbb{R})} \leq \tilde{K}_1 \varepsilon,
\end{equation}
for some $\tilde{K}_1>0$.
Define the one dimensional quadratic form
\begin{equation*}
	\widetilde{n}(w) :=
	\int_\mathbb{R}\left(|w'|^2 
	+ \left(k\cdot\Theta - \frac{|\Theta'|^2}{2} \right) |w|^2
	\right) \ds \dt 
	+ \boldsymbol{\tilde{\kappa}} \int_\mathbb{R} |w|^2 \ds \dt,
\end{equation*}
$\dom n = H^1(\mathbb{R})$, and denote by $\widetilde{N}$ its associated self-adjoint operator.
Following the same steps of the proof of Proposition \ref{propalmfin}, and by the estimates in (\ref{label01}) and (\ref{label02}), we obtain
\begin{equation}\label{convend}
	\left\|\widetilde{L}_\varepsilon^{-1} - \widetilde{N}^{-1} \oplus{\bf 0}\right\| \leq \tilde{K}_{2}\varepsilon^{1/2},
\end{equation}
for some $\tilde{K}_{2}$.
Recall that ${\bf 0}$ is the null operator on the subspace ${\cal A}^\perp$.

\begin{proof}[\bf Proof of Theorem \ref{lambdaassintotico}]
	Follows from estimate (\ref{convend}).
\end{proof}

\section{Appendix}\label{appendix001}

In this appendix we present some results already known in the literature that were cited in this work. 

\begin{Lemma}\label{lemmakriz}
Let $\alpha\in\mathbb{R}$ and $V$ be  a real measurable bounded
function on $(0,1)$. Consider the self-adjoint operator
$H_\alpha := -\partial_t^2 + V(t)\Id$, $\dom H_\alpha := \{\psi\in H^2(0,1); \psi(0)=0, \psi'(1)+\alpha\psi(1)=0\}$,
acting in $L^2(0,1)$. 
Let $E_1(\alpha)$ be the first eigenvalue of $H_\alpha$. 
If $\alpha_1 \leq \alpha_2$, then $E_1(\alpha_1) \leq E_1(\alpha_2)$. More precisely,
\[E_1(\alpha_1)\leq E_1(\alpha_2) + (\alpha_1-\alpha_2)\frac{\psi_1(1)^2}{\|\psi_1\|^2},\]
where $\psi_1$ is a real eigenfunction corresponding to $E_1(\alpha_1)$.
\end{Lemma}

\begin{Theorem}\label{limiteautovalores}
Let $\mu \in \mathbb{R}\backslash\{0\}$ and $V$ be  a real measurable bounded
function on $\mathbb{R}$.
Consider the self-adjoint operator 
$H_\mu := -\Delta_\mathbb{R}+ \mu V(s)\Id,$ 
$\dom H_\mu := H^2(\mathbb{R}),$
acting in $L^2(\mathbb{R})$. 
Then, for each $j\in \mathbb{N}$, 
\[ \lim_{\mu \to +\infty}\frac{\lambda_j(H_\mu)}{\mu} = \inf_{s\in \mathbb{R}}{V(s)}.\]
\end{Theorem}

\noindent
The proofs of Lemma \ref{lemmakriz} and Theorem \ref{limiteautovalores} can be found in \cite{davidkriz} and \cite{freitas}, respectively, which will skipped for being beyond the scope of this paper.


\vspace{0.2cm}
\noindent
{\bf Acknowledgments}

\vspace{0.2cm}
\noindent
Rafael T. Amorim was supported by CNPq (Brazil), grant 141842/2019-9.

\end{document}